\documentclass[reqno,11pt]{amsart}
\usepackage{amsmath,amsfonts,amsthm,amssymb,mathrsfs}
\usepackage{setspace,comment,verbatim}
\usepackage{fancyhdr}
\usepackage{lastpage}
\usepackage{extramarks}
\usepackage{chngpage}
\usepackage{soul,color}
\usepackage{graphicx,float,wrapfig}
\DeclareMathSymbol{\R}{\mathbin}{AMSb}{"52}

\newcommand{\E}{\mathbb{E}}
\newcommand{\F}{\mathscr{F}}
\renewcommand{\H}{\mathscr{H}}

\renewcommand{\R}{\mathbb{R}}
\newcommand{\inv}{\text{inv}}
\newcommand{\Inv}{\text{Inv}}
\newcommand{\Div}{\text{Cinv}}
\newcommand{\cross}{\text{cross}}
\newcommand{\nest}{\text{nest}}
\newcommand{\Cross}{\text{Cross}}
\newcommand{\Nest}{\text{Nest}}
\renewcommand{\div}{\text{cinv}}
\DeclareMathSymbol{\N}{\mathbin}{AMSb}{"4E}
\DeclareMathSymbol{\Z}{\mathbin}{AMSb}{"5A}

\renewcommand{\sp}{\vspace{10pt}}

\newtheorem*{thm*}{Theorem}
\newtheorem{lem}{Lemma}
\newtheorem{remark}{Remark}
\newtheorem{defn}{Definition}
\newtheorem{defthm}{Definition/Theorem}
\newtheorem{cor}{Corollary}
\newtheorem{prop}{Proposition}
\begin{document}
\title{The $(q,t)$-Gaussian Process}
\author{Natasha Blitvi\'c\\}

\begin{abstract}
The\thanks{\footnotesize The question underlying this work arose while the author was visiting her advisor, T. Kemp, at the University of California in San Diego. The author's research was supported by the \emph{Claude E. Shannon Research Assistantship} at MIT and the \emph{Chateaubriand Fellowship} at the Institut Gaspard Monge of the Univerist\'e Paris-Est, co-advised by Ph. Biane.\\
{\it Address 1:} Massachusetts Institute of Technology, Room 2-341, 77 Massachusetts Avenue, Cambridge, MA 02139.\\
{\it Address 2:} Universit\'e Paris-Est, LIGM, Copernic 4B03R, 77454 Marne-la-Vall\'ee cedex 2, France.\\
{\it Email:} blitvic@mit.edu} $(q,t)$-Fock space $\F_{q,t}(\H)$, introduced in this paper, is a deformation of the $q$-Fock space of Bo\.zejko and Speicher. The corresponding creation and annihilation operators now satisfy the commutation relation
$$a_{q,t}(f)a_{q,t}(g)^\ast-q \,a_{q,t}(g)^\ast a_{q,t}(f)= \langle f,g\rangle_{_\H}\, t^{N},$$
with $N$ denoting the usual number operator, and generate a Hilbert space representation of the Chakrabarti-Jagannathan deformed quantum oscillator algebra. The moments of the deformed field operator $s_{q,t}(h):=a_{q,t}(h)+a_{q,t}(h)^\ast$, the present analogue of the Gaussian random variable, are encoded by the joint statistics of crossings and nestings in pair partitions. The restriction of the vacuum expectation state to the $(q,t)$-Gaussian algebra is not tracial for $t\neq 1$.  

The $q=0<t$ specialization yields a new single-parameter deformation of the full Boltzmann Fock space of free probability. This refinement is particularly natural as the probability measure associated with the deformed semicircular element turns out to be encoded via the Rogers-Ramanujan continued fraction, the $t$-Airy function, the $t$-Catalan numbers of Carlitz-Riordan, and the first-order statistics of the reduced Wigner process.\\


\end{abstract}

\maketitle

\section{Introduction}
\emph{Non-commutative probability} refers, in broad terms, to a generalization of the classical probability theory to encompass quantum observables. In particular, Kolmogorov's probability triple is replaced by a \emph{non-commutative probability space} $(\mathscr A,\varphi)$, where $\mathscr A$ is a unital $\ast$-algebra whose elements are interpreted as \emph{non-commutative random variables} and the unital linear functional $\varphi$ plays the role of the expectation. Particularly relevant examples are to be found among algebras of linear operators on the Bosonic and Fermionic Fock spaces together with the corresponding vacuum expectation states. The Bosonic and Fermionic annihilation operators, $\{a_i^+\}$ and $\{a_i^-\}$ respectively, together with their adjoints (the creation operators) satisfy the commutation relations
\begin{equation}(a_i^+)(a_j^+)^\ast-(a_j^+)^\ast(a_i^+)=\delta_{i,j}\tag{CCR}\end{equation}
\begin{equation}(a_i^-)(a_j^-)^\ast+(a_j^-)^\ast(a_i^-)=\delta_{i,j}\tag{CAR},\end{equation}
where (CCR) is referred to as the \emph{canonical commutation relations}\footnote{It should be noted that the 
operators participating in CCR are in fact \emph{not bounded} and do not live in a $\ast$-algebra. Rather, they are affiliated with one and the corresponding $\ast$-algebra can be identified with $L^\infty$ of the Gaussian measure. All other examples throughout this paper will take place in $\ast$-algebras of bounded operators.} and (CAR) as the \emph{canonical anti-commutation relations}.

In the setting of (CCR), of particular note is the fact that the Bosonic field operator $s_1=a_1+a_1^\ast$ can be naturally identified with the classical Gaussian random variable.\footnote{In the sense that for any continuous function $f$ supported on the spectrum of $s_1$, $\int fd\mu=\varphi(f(s_1))$ with  $d\mu(x)=(2\pi)^{-1/2}e^{-x^2/2}dx$. Equivalently, 
$\varphi(s_1^n)$ equals the $n^\text{th}$ moment of the standardized Gaussian random variable.}
It is in this context that Frisch and Bourret \cite{Frisch1970} first considered a deformation of the canonical (anti-)commutation relations, as
$a_ia_j^\ast-qa_j^\ast a_i=\gamma_{i,j}\delta_{i,j},$
for $q\in [-1,1]$ and $\gamma_{i,j}$ some positive-definite function, and studied the properties of the ``parastochastic'' random variable $a_i+a_i^\ast$. However, it was only two decades later, in Bo\.zejko and Speicher's independent study of deformed commutation relations \cite{Bozejko1991}, that the existence question was resolved and such processes constructed. (For existence proofs, see also \cite{Bozejko1994, Fivel1992, Greenberg1991, Speicher1992, Speicher1993,Yu1994, Zagier1992}.) In particular, Bo\.zejko and Speicher constructed a suitably deformed Fock space on which the creation and annihilation operators (analogues of those on the classical Fock spaces) satisfied the deformed commutation relation:
\begin{equation}a_ia_j^\ast-qa_j^\ast a_i=\delta_{i,j}.\tag{$q$-CR}\end{equation}
The setting for these operators is provided by a direct sum of single particle spaces
\begin{equation}\mathscr F=(\mathbb C \Omega)\oplus\bigoplus_{n\geq 1} \H_{\mathbb C}^{\otimes n}\label{F}\end{equation}
with $\Omega$ being some distinguished unit vector and $\H_{\mathbb C}$ the complexification of some real separable Hilbert space $\H$. On this space, one can define the deformed sesquilinear form $\langle \xi,\eta\rangle_q=\langle \xi,P_q\eta\rangle_0,$
where $\langle\;,\;\rangle_0$ is the usual inner product on the full Fock space (cf. Section~\ref{secqtFock}). 
The particularly insightful choice of a \emph{positive} ``projection" operator $P_q$, which would turn $\langle\;,\;\rangle_q$ into a bona fide inner product and have the classical creation operator and its resulting adjoint satisfy ($q$-CR), was the key element in Bo\.zejko and Speicher's construction.
Consider the unitary representation $\pi\mapsto U_\pi^{(n)}$ of the symmetric group $S_n$ on $\H^{\otimes n}$, given by
$U_\pi^{(n)} \,h_1\otimes\ldots\otimes h_n=h_{\pi(1)}\otimes\ldots\otimes h_{\pi(n)}.$
Writing
$P_{q}=\bigoplus_{n=0}^\infty P_{q}^{(n)}$ with $P_{q}^{(n)}:\H^{\otimes n}\to \H^{\otimes n},$
the key idea of Bo\.zejko and Speicher was to replace the (anti-)symmetrization step that would yield the two classical inner products by instead ``$q$-symmetrizing'' as
\begin{equation}P_{q}^{(n)}:=\sum_{\pi\in S_n} q^{\inv(\pi)}\,U_\pi^{(n)},\end{equation}
where $\inv(\pi)$ gives the number of inversions in a permutation $\pi$ (cf. Section~\ref{preliminaries}). By showing that $P_{q}$ was indeed positive definite for $|q|<1$ (or positive semi-definite for $|q|\leq 1$) and completing (and separating) $\mathscr F$ with respect to the resulting inner product $\langle\;,\;\rangle_q$, Bo\.zejko and Speicher thus constructed and introduced the so-called \emph{$q$-Fock space} $\mathscr F_q$.

Following much interest over the course of the two subsequent decades, the $q$-Fock space is now known to have a number of remarkable properties. Far from attempting to overview them all, we now focus on the probabilistic aspects of the algebras of bounded linear operators and point out a few key results. Namely, starting with the very motivation, it was shown in \cite{Bozejko1991} that the deformed field operator $s_q(h):= a_q(h)^\ast+a_q(h)$ is a natural deformation, in the setting of ($q$-CR), of the Gaussian random variable. The corresponding analogues of the complex Gaussian random variable were also constructed \cite{Mingo2001, Kemp2005}. More generally, an analogous formulation was found to give rise to a $q$-deformation of the Brownian motion \cite{Bozejko1991, Bozejko1997} and, even more broadly, to a characterization of the $q$-L\'evy processes \cite{Anshelevich2001}. Yet, there is also much about the structure of the operator algebras on the $q$-Fock spaces that remains mysterious and is object of current research. For example, remaining in the probabilistic vein, it was very recently shown that the $q$-Gaussian random variables are (somewhat unexpectedly) freely infinitely divisible \cite{Anshelevich2010}.

The goal of this paper is to introduce a second-parameter refinement of the $q$-Fock space, formulated as the \emph{$(q,t)$-Fock space} $\mathscr F_{q,t}$. The $(q,t)$-Fock Space is constructed via a direct refinement of Bo\.zejko and Speicher's framework \cite{Bozejko1991}, yielding the $q$-Fock space when $t=1$. Before overviewing the details of the construction and the main results, we take a moment to point out why the present refinement is, in fact, particularly natural. 
\begin{itemize}
\item Starting with the structural aspect and considering the combinatorial framework underlying the $q$-Fock space, the notions of permutation inversions and crossings in pair partitions, which play the key role in the original formulation, are now replaced by the joint statistics given by permutation \emph{inversions and co-inversions} and \emph{crossings and nestings} in pair partitions (cf. Section~\ref{preliminaries} and Section~\ref{secqtSemicircular}). In particular, while the permutation co-inversions are the complements of the permutation inversions, a satisfying characterization of the joint distribution of crossings and nestings of pair partitions is an open problem in combinatorics, of relevance to broader combinatorial questions \cite{Kasraoui2006,Klazar2006,Chen2007}. 
\item Whereas the distribution of the $q$-Gaussian operator is the unique measure that orthogonalizes the $q$-Hermite orthogonal polynomials, given by the three-term recurrence $zH_n(z;q,t)=H_{n+1}(z;q,t)+[n]_{q}H_{n-1}(z;q,t)$, the \emph{$(q,t)$-Gaussian} orthogonalizes the \emph{$(q,t)$-Hermite orthogonal polynomials}. The latter are given by the recurrence $zH_n(z;q,t)=H_{n+1}(z;q,t)+[n]_{q,t}H_{n-1}(z;q,t)$. (cf. Definition/Theorem~\ref{defthm3} below.)
\item The $q=0<t$ case corresponds to a new single-parameter deformation of the full Boltzmann Fock space of free probability \cite{Voiculescu1986,Voiculescu1992} and of the corresponding semicircular operator. The corresponding measure is encoded, in various forms, via the Rogers-Ramanujan continued fraction (e.g. \cite{Andrews}), the Rogers-Ramanujan identities (e.g. \cite{Andrews}), the $t$-Airy function \cite{Ismail2005}, the $t$-Catalan numbers of Carlitz-Riordan \cite{Furlinger1985,Carlitz1964}, and the first-order statistics of the reduced Wigner processes \cite{Khorunzhy,Mazza2002}.
\end{itemize}

At this point, it is also important to note that the $(q,t)$-deformed framework independently arises in a more general asymptotic setting, through non-commutative central limit theorems and random matrix models. This is the subject of the companion paper \cite{Blitvic2}. In particular, analogously to the $q$-annihilation and creation operators which arise as weak limits in the Non-commutative Central Limit Theorem introduced by Speicher \cite{Speicher1992}, the $(q,t)$-annihilation and creation operators appear in the general form of the theorem developed in \cite{Blitvic2}. In broad terms, this second-parameter refinement is a consequence of the passage from a commutation structure built around commutation \emph{signs}, taking values in $\{-1,1\}$, to a more general structure based on commutation coefficients taking values in $\mathbb R$.

However, at this point it should also be remarked that, despite many analogies with the original formulation, the $(q,t)$-Fock space has a surprising property that distinguishes it from the $q$-Fock space when $t\neq 1$. Namely, the vacuum expectation state $\varphi$ (cf. Section~\ref{secqtFock}) is not tracial on the $\ast$-algebra generated by the field operators $\{s_{q,t}(h)\}_{h\in\H}$.

\vspace{10pt}

The following is an overview of the main results in this paper, encompassing an overview of the $(q,t)$-Fock space construction.

\begin{defthm} Let $\pi\mapsto U_\pi^{(n)}$ denote the unitary representation of the symmetric group $S_n$ on $\H^{\otimes n}$ given by
$U_\pi^{(n)} \,h_1\otimes\ldots\otimes h_n=h_{\pi(1)}\otimes\ldots\otimes h_{\pi(n)}$ and, for every permutation $\pi\in S_n$, let $\inv(\pi)$ and $\div(\pi)$ respectively denote the inversions and co-inversions of a permutation (cf. Section~\ref{preliminaries}). Given the vector space $\mathscr F$ given in (\ref{F}), consider the ``projection'' operator $P_{q,t}:\F\to\F$ given by
$P_{q,t}=\bigoplus_{n=0}^\infty P_{q,t}^{(n)}$ with $P_{q,t}^{(n)}:\H^{\otimes n}\to \H^{\otimes n},$
with
\begin{equation}P_{q,t}^{(n)}:=\sum_{\pi\in S_n} q^{\inv(\pi)}t^{\div(\pi)}\,U_\pi^{(n)}.\end{equation}
Consider, further, the sesquilinear form $\langle\;,\;\rangle_{q,t}$ on $\mathscr F$ given, via the usual inner product $\langle\;,\;\rangle_0$ on the full Fock space (cf. Section~\ref{secqtFock}), by $\langle\eta,\xi\rangle_{q,t}=\langle\xi,P_{q,t}\,\eta\rangle_0\quad\forall\,\eta,\xi\in\mathscr F.$

Then, for all $n\in\mathbb N$, $P_{q,t}^{(n)}$ is strictly positive definite for all $|q|<t$, $\langle\;,\;\rangle_{q,t}$ 
is an inner product, and the \emph{$(q,t)$-Fock space} is the completion of $\mathscr F$ with respect to the norm induced by $\langle\;,\;\rangle_{q,t}$.
\end{defthm}

\begin{defthm} Given the $(q,t)$-Fock space $\F_{q,t}$ and the underlying real Hilbert space $\H$, define the respective \emph{creation} and \emph{annihilation} operators $a_{q,t}(f)^\ast$ and $a_{q,t}(f)$ on $\F_{q,t}$, for $f\in H$ and $|q|< t\leq 1$, by linear extension of:
\begin{equation}a_{q,t}(f)^\ast\Omega = f,\quad\quad a_{q,t}(f)^\ast h_1\otimes\ldots\otimes h_n=f\otimes h_1\otimes\ldots\otimes h_n\label{cdef1}\end{equation}
and
\begin{equation}a_{q,t}(f)\Omega = 0, \quad\quad a_{q,t}(f)h_1\otimes\ldots\otimes h_n=\sum_{k=1}^n q^{k-1}t^{n-k} \langle f,h_k\rangle_{_\H}\,  h_1\otimes \ldots\otimes\breve h_k\otimes\ldots\otimes h_n.\label{cdef2}\end{equation}

Then, $a_{q,t}(f)$ and $a_{q,t}(f)^\ast$ are adjoints with respect to $\langle\;,\;\rangle_{q,t}$. Further, for all $f,g\in \H$, the operators satisfy the $(q,t)$-commutation relation
\begin{equation}a_{q,t}(f)a_{q,t}(g)^\ast-q \,a_{q,t}(g)^\ast a_{q,t}(f)= \langle f,g\rangle_{_\H}\, t^{N},\tag{$(q,t)$-CR}\end{equation}
where $t^N$ is the operator on $\F_{q,t}$ defined by the linear extension of $t^N\Omega=\Omega$ and $t^N h_1\otimes\ldots\otimes h_n=t^n h_1\otimes\ldots\otimes h_n$ for all $h_1,\ldots,h_n\in \H$. 
Moreover, for all $n\in\mathbb N$ and $\epsilon(1),\ldots,\epsilon(2n)\in\{1,\ast\}$, the corresponding mixed moment of the creation and annihilation operators on $\F_{q,t}$ is given by
\begin{align}\displaystyle&\varphi_{q,t}(a_{q,t}(h_1)^{\epsilon(1)}\ldots a_{q,t}(h_{2n-1})^{\epsilon(2n-1)})=0\\\nonumber
\\&\varphi_{q,t}(a_{q,t}(h_1)^{\epsilon(1)}\ldots a_{q,t}(h_{2n})^{\epsilon(2n)})=\sum_{\mathscr V\in \mathscr P_{2}(2n)}q^{\cross(\mathscr V)}t^{\nest(\mathscr V)}\varphi(a_{q,t}(h_{w_1})^{\epsilon(w_1)}a_{q,t}(h_{z_1})^{\epsilon(z_1)})\ldots\nonumber\\&\hspace{7cm}\ldots\varphi(a_{q,t}(h_{w_n})^{\epsilon(w_n)}a_{q,t}(h_{z_n})^{\epsilon(z_n)}),\end{align}
where $P_{2}(2n)$ is the collection of pair partitions of $[2n]$ (cf. Section~\ref{preliminaries}), each $\mathscr V\in P_{2}(2n)$ is (uniquely) written as a collection of pairs $\{(w_1,z_1),\ldots,(w_n,z_n)\}$ with $w_1<\ldots<w_n$ and $w_i<z_i$, with $\cross(\mathscr V)$ and $\nest(\mathscr V)$ denoting the numbers of crossings and nestings, respectively, in $\mathscr V$ (cf. Section~\ref{preliminaries}). Finally, for any $f\in\H$, the operator $a_{q,t}(f)$ on $\F_{q,t}$ is bounded for $0\leq |q|< t\leq 1$, with norm given by
\begin{equation}\displaystyle\|a_{q,t}(f)\|=\|a_{q,t}(f)^\ast\|=\left\{\begin{array}{ll}\|f\|_\H& 0\leq -q< t\leq 1,\\&\\\frac{1}{\sqrt{1-q}}\,\|f\|_\H&0<q<t=1\\&\\\sqrt{\frac{t^{n_\ast}-q^{n_\ast}}{t-q}}\,\|f\|_\H&0<q<t<1\end{array}\right.,\end{equation}
for
\begin{equation}n_\ast=\left\lceil\frac{\log\left(1-q\right)-\log\left(1-t\right)}{\log\left(t\right)-\log\left(q\right)}\right\rceil.\end{equation}
\end{defthm}

\begin{defthm} For $h\in\H$, the \emph{$(q,t)$-Gaussian} element $s_{q,t}(h) \in \mathscr B(\mathscr F_{q,t})$ is given by $s_{q,t}(h):=a_{q,t}(h)+a_{q,t}(h)^\ast$. The $(q,t)$-Gaussian is self-adjoint, with moments given by
\begin{eqnarray}
\varphi_{q,t}(s_{q,t}(h)^{2n-1})&=&0\\
\varphi_{q,t}(s_{q,t}(h)^{2n})&=&\|h\|_\H^{2n} \sum_{\mathscr V\in \mathscr P_2(2n)}  \!\!\!q^{\cross(\mathscr V)}\,t^{\nest(\mathscr V)}
=\,\|h\|_\H^{2n}\,\,[z^n]\,\cfrac{1}{1-\cfrac{[1]_{q,t}z}{1-\cfrac{[2]_{q,t}z}{1-\cfrac{[3]_{q,t}z}{\ldots}}}},
\end{eqnarray}
where $[z^n](\cdot)$ denotes the coefficient of the $z^n$ term in the power series expansion of $(\cdot)$ and for all $n\in\mathbb N$,
\begin{equation}[n]_{q,t}:=\sum_{i=1}^n q^{i-1}t^{n-i}=\frac{t^n-q^n}{t-q}.\end{equation} Furthermore, the distribution of $s_{q,t}(e)$, with $\|e\|_\H=1$, is the unique real probability measure that orthogonalizes the \emph{$(q,t)$-Hermite orthogonal polynomial sequence} given by the recurrence 
\begin{equation}zH_n(z;q,t)=H_{n+1}(z;q,t)+[n]_{q,t}H_{n-1}(z;q,t),\label{eqrecurrenceH}\end{equation}
with
$H_0(z;q,t)=1,\,\,H_1(z;q,t)=z.$
\label{defthm3}
\end{defthm}

\begin{defthm} For $q=0<t\leq 1$, $\mathscr B(\F_{0,t})$ is given as the von Neumann algebra generated by $\{a_{q,t}(h)\}_{h\in\H}$. The \emph{$t$-semicircular} element is the $q=0<t\leq 1$ specialization of the $(q,t)$-Gaussian element $s_{0,t}(h)\in \mathscr B(\mathscr F_{0,t})$. The moments of the $t$-semicircular element are given by
\begin{eqnarray*}
\varphi_{0,t}(s_{0,t}(h)^{2n-1})&=&0\\
\varphi_{0,t}(s_{0,t}(h)^{2n})&=&\|h\|_\H^{2n} \sum_{\mathscr V\in NC_2(2n)} t^{\nest(\mathscr V)}=\|h\|_\H^{2n}\,C_n^{(t)}
\end{eqnarray*}
where $NC_2(2n)$ denotes the lattice of non-crossing pair-partitions and $C_n^{(t)}$ are referred to as the Carlitz-Riordan $t$-Catalan  numbers \cite{Furlinger1985,Carlitz1964}, given by the recurrence
\begin{equation}
C_n^{(t)}=\sum_{k=1}^n t^{k-1} C_{k-1}^{(t)} C_{n-k}^{(t)},\label{qCat}
\end{equation}
with $C_0^{(t)}=1$.
The moments of the normalized $t$-semicircular element $s_{0,t}:=s_{0,t}(e)$, for $\|e\|_\H=1$, are encoded by the Rogers-Ramanujan continued fraction as
\begin{equation}\sum_{n\geq 0}\varphi_{q,t}(s_{q,t}(h)^{n})z^n=\cfrac{1}{1-\cfrac{t^0\,z}{1-\cfrac{t^1\,z}{1-\cfrac{t^2\,z}{\ldots}}}}.\end{equation}
Furthermore, the Cauchy transform of the corresponding normalized $t$-semicircular measure $\mu_{0,t}$ is given by 
\begin{equation}\int_{\R}\frac{1}{z-\eta}\,d\mu_{0,t}(\eta)=\frac{1}{z}\,\frac{\sum_{n\geq 0} (-1)^n\,\frac{t^{n^2}}{(1-t)(1-t^2)\ldots(1-t^n)}\,z^{-n}}{\sum_{n\geq 0} (-1)^n\,\frac{t^{n(n-1)} }{(1-t)(1-t^2)\ldots(1-t^n)}\,z^{-n}}=\frac{1}{z}\,\frac{A_t(1/z)}{A_t(1/(zt))},\end{equation}
where $A_t$ denotes the $t$-Airy function of \cite{Ismail2005}, given by
\begin{equation}
A_t(z)=\sum_{n\geq 0}\frac{t^{n^2}}{(1-t)\ldots(1-t^n)}(-z)^n.
\end{equation}
In particular, letting $\{z_j\}_{j\in\mathbb N}$ denote the sequence of zeros of the rescaled $t$-Airy function $A_t(z/t)$, the measure $\mu_{0,t}$ is a discrete probability measure with atoms at
$$\pm \sqrt{t/z_j},\quad j\in\mathbb N$$
with corresponding mass
$$-\frac{A_t(z_j)}{2\,z_j\,A_t^\prime(z_j/t)},$$
where $A_t^\prime(z):=\frac{d}{dz}A_t(z)$. The only accumulation point of $\mu_{0,t}$ is the origin.

The normalized $t$-semicircular measure $\mu_{0,t}$ is also the unique probability measure orthogonalizing the \emph{$t$-Chebyshev II orthogonal polynomial sequence} $\{U_n(z;t)\}_{n\geq 0}$, a specialization of the orthogonal polynomials of Al-Salam and Ismail \cite{Al-Salam1983} and determined by the three-term recurrence
$$zU_n(z;t)=U_{n+1}(z;t)+t^{n-1}U_{n-1}(z;t),$$
with
$$U_0(z;t)=1,\,\,U_1(z;t)=z.$$

Finally, $t$-semicircular measure $\mu_{0,t}$ is, in a certain sense, the weak limit of the first-order statistics of the reduced Wigner process \cite{Khorunzhy,Mazza2002}. Specifically, for all $\rho\in[0,1]$ and $n\in\mathbb N$,
\begin{equation}\lim_{N\to\infty}\varphi_N\left(\frac{W_{N,\rho}(1)}{N}\,\frac{W_{N,\rho}(2)}{N}\ldots \frac{W_{N,\rho}(n)}{N}\right)=\rho^{n/2}\,\varphi_{0,t}(s_{0,t}^n)\quad\quad\text{for}\quad t=\rho^2,\end{equation}
where $\varphi_N=\frac{1}{N} \text{Tr}\otimes \E$ and $\{W_{N,\rho}(k)\}_{k\in\mathbb N}$ is a sequence of Wigner matrices with correlations
\begin{eqnarray}
\E(w_{i,j}(k)w_{i',j'}(k))&=&\left\{\begin{array}{ll}1,&(i,j)=(i',j')\text{ or }(i,j)=(j',i')\\0,&\text{otherwise}\end{array}\right.\\\E(w_{i,j}(k)w_{i,j}(m))&=&\rho^{m-k}\quad\text{ for }m>k,
\end{eqnarray}
with $w_{i,j}(k)$ denoting the $(i,j)^\text{th}$ entry of $W_{N,\rho}(k)$.
\end{defthm}

Finally, recall that the crux of the companion paper \cite{Blitvic2} is the extension of Speicher's \emph{Non-commutative Central Limit Theorem} \cite{Speicher1992}, giving rise to an \emph{an asymptotic model for operators satisfying the commutation relation $(q,t)$-CR}. It should be emphasized that this asymptotic model provides an existence proof, independent of the explicit construction of Section~\ref{secqtFock}, and provides an alternative reason for the fundamental ordering bound present throughout, namely $|q|<t$.

\subsection{Deformed Quantum Harmonic Oscillators}
In physics, the oscillator algebra of the quantum harmonic oscillator is generated by elements $\{1, a, a^\ast,N\}$ satisfying the canonical commutation relations
$$[a, a^\ast] = 1,\quad [N, a] = -a, \quad [N, a^\ast] = a^\ast,$$
where $a^\ast$, $a$, and $N$ can be identified with the creation, annihilation, and number operators on the Bosonic Fock space. Physicists may also speak of \emph{generalized deformed oscillator algebras}, which are instead generated by elements satisfying the deformed relations
$$a^\ast a = f(N),\quad aa^\ast = f(N + 1),\quad [N, a] = -a,\quad [N, a^\ast] = a^\ast,$$
where $f$ is typically referred to as the \emph{structure function of the deformation}.
While an in-depth review of single-parameter deformations of the quantum oscillator algebra is available in \cite{1464-4266-4-1-201}, of particular interest are the so-called Arik-Coon $q$-deformed oscillator algebra \cite{Arik1976} given by
$$aa^\ast - qa^\ast a = 1,\quad [N, a] = -a, [N, a^\ast] = a^\ast, \quad f(n)=\frac{1-q^n}{1-q},\quad q\in\R_+$$
and the Biedengarn-Macfarlane $q$-deformed oscillator algebra \cite{0305-4470-22-18-004,0305-4470-22-21-020} given by
$$aa^\ast - qa^\ast a = q^{-N},\quad [N, a] = -a,\quad [N, a^\ast] = a^\ast,\quad f(n)=\frac{q^{-n}-q^n}{q^{-1}-q}, \quad q\in \R_+.$$
For $q\in [-1,1]$, the Hilbert space representation of the Arik-Coon algebra, generalized to an infinite-dimensional setting, is given by the $q$-Fock space of Bozejko and Speicher \cite{Bozejko1991}. Manipulating Definition/Theorem 1 and 2, the reader may readily verify that for $q\in (0,1]$, a Hilbert space realization of the Biedengarn-Macfarlane algebra is similarly given by the $(q,t)$-Fock space specialized to $t=q^{-1}$. More generally, the Chakrabarti-Jagganathan oscillator algebra \cite{0305-4470-24-13-002} is given by
$$aa^\ast - qa^\ast a = p^{-N},\quad [N, a] = -a,\quad [N, a^\ast] = a^\ast,\quad f(n)=\frac{p^{-n}-q^n}{p^{-1}-q}, \quad q,p\in \R_+.$$
For $0\leq |q|<p^{-1}$, the reader may also verify that the $(q,t)$-Fock space again provides the desired Hilbert space realization for for $t=p^{-1}$. A Bargmann-Fock representation of this algebra was previously considered in \cite{Burban1993} for the case of a single oscillator, but no explicit underlying Hilbert space was constructed nor shown to exist in the parameter range considered. Instead, an explicit representation was provided in the space of analytic functions via suitably deformed differential operators. Note that the two-parameter deformation of the Hermite orthogonal polynomial sequence given by (\ref{eqrecurrenceH}) also appears as recurrence (15) in \cite{Burban1993}.

\subsection{General Brownian Motion}
From a high-level perspective, the $(q,t)$-Gaussian processes fall under the framework of \emph{Generalized Brownian Motion} \cite{Bozejko1996}. The latter is described by families of self-adjoint operators $G(f)$, where $f$ belongs to some real Hilbert space $\H$, and a state $\varphi$ on the algebra generated
by the $G(f)$, given by
$$\varphi(G(f_1)\ldots G(f_{2n}))=\sum_{\mathscr V\in\mathscr P_2(2n)}\tau(\mathscr V)\prod_{(i,j)\in\mathscr V}\langle f_i,f_j\rangle_\H,$$
for some positive definite function $\tau:\mathscr P_2(2n)\to\R$. Thus, the $q$-Brownian motion is given by $\tau_q(\mathscr V) = q^{\cross(\mathscr V)}$, whereas in the present context, the $(q,t)$-Brownian motion corresponds to $\tau_{q,t}(\mathscr V) = q^{\cross(\mathscr V)}t^{\nest(\mathscr V)}$. Many other known $\tau$ functions exist. Most generally, a beautiful framework by Gu{\c{t}}{\u{a}} and Maassen \cite{Guta2001,Guta2001b}, proceeding via the combinatorial theory of species of structures~\cite{Joyal1981}, encompasses the familiar deformations (Bosonic, Fermionic, free, as well as \cite{Bozejko1991,Bozejko1996}). It is foreseeable, though not presently clear, whether (or how) the current formulation can be encompassed within the same framework.

Arguably the most closely-related framework to that of the $(q,t)$-Gaussians is provided by Bo\.zejko in \cite{Bozejko2007}, introduced with the goal of extending the $q$-commutation relations beyond the $q\in[-1,1]$ parameter range. In particular, Bo\.zejko studied operators satisfying two types of deformed commutation relations, namely
$$A(f)A^\ast(g)-A^\ast(g)A(f)=q^N\langle f,g\rangle_\H\quad \text{for } q>1$$
and 
$$B(f)B^\ast(g)+B^\ast(g)B(f)=|q|^N\langle f,g\rangle, \quad\text{for } q<-1.$$
Comparing Bo\.zejko's creation and annihilation operators and his $q$-deformed inner product with the present definitions, Bo\.zejko's setting turns out to be that of the $(q,t)$-Fock space for $q\mapsto 1$ an $t\mapsto q$. The difficulty encountered by Bo\.zejko, pertaining to the fact that the Gaussian element is no longer self-adjoint, is consistent with the fact that for $t>1$, the creation and annihilation operators are no longer bounded.

As later found, a relevant two-parameter deformation of the canonical commutation and anti-commutation relations was previously studied by Bo\.zejko and Yoshida \cite{Bozejko2006}, as part of a more general framework. In \cite{Bozejko2006}, a general version of the $n$-dimensional ``projection'' operator $P_{q,t}^{(n)}$ was given by $\prod_{i=1}^n\tau_i\, \sum_{\pi\in S_n} q^{\inv(\pi)}$, where $\{t_n\}$ is any sequence of positive numbers. The reader may readily verify that the $(q,t)$-projection operator $P_{q,t}^{(n)}$ is recovered for $\tau_n=t^{n-1}$ and by substituting $q\mapsto q/t$. Instead, Bo\.zejko and Yoshida focus on the specialization $\tau_n=s^{2n}$ and $q\mapsto q$, yielding the so-called \emph{$(q,s)$-Fock space}. The corresponding combinatorial structure (now given in terms of crossings and \emph{inner points}), the Wick-type formulas, and continued fractions are considered in \cite{Bozejko2006}.

In the context of the generalized Brownian motion, the point of the present pair of articles, which construct and describe various aspects of the $(q,t)$-Fock space, is to argue that the framework at hand is ultimately a highly natural refinement of the $q$-Fock space. This argument is based on the structural depth of the $(q,t)$-deformed framework, as evidenced by its intimate ties to various fundamental mathematical objects, as well as to the correspondence with the natural generalizations of non-commutative asymptotic frameworks (developped in \cite{Blitvic2}).
\vspace{10pt}

\section{Combinatorial Preliminaries}
\label{preliminaries}
The present section briefly overviews the key combinatorial constructs that will, in the subsequent sections, be liberally aplied to encode the structure of $(q,t)$-Fock spaces, the creation, anihilation, and field operators on these spaces, as well as the mechanics of the relevant non-commutative limit theorems. In broad terms, the objects of interest are \emph{set partitions}, \emph{permutations}, \emph{lattice paths}, and certain combinatorial statistics thereof.

\subsection{Partitions}
Denote by $\mathscr P(n)$ the set of partitions of $[n]:=\{1,\ldots,n\}$ and by $\mathscr P_2(2n)$ the set of \emph{pair partitions} of $[2n]$, that is, the set of paritions of $[2n]$ with each part containing exactly two elements. (Note that the pair partitions are also referred to as \emph{pairings} or \emph{perfect matchings}.) It will be convenient to represent a pair partition as a list of ordered pairs, that is, $\mathscr P_2(2n)\ni\mathscr V=\{(w_1,z_1),\ldots,(w_n,z_n)\}$, where $w_i<z_i$ for $i\in[n]$ and $w_1<\ldots<w_n$.

Of particular interest are the following two statistics on $\mathscr P_2(2n)$.
\begin{defn}[Crossings and Nestings] For $\mathscr V=\{(w_1,z_1),\ldots,(w_{n},z_{n})\}\in\mathscr P_2(2n)$, pairs $(w_i,z_i)$ and $(w_j,z_j)$ are said to \emph{cross} if $w_i<w_j<z_i<z_j$. The corresponding \emph{crossing} is encoded by $(w_i,w_j,z_i,z_j)$ with 
\begin{eqnarray*}\Cross(\mathscr V)&:=&\{(w_i,w_j,z_i,z_j)\mid (w_i,z_i),(w_j,z_j)\in\mathscr V\text{ with } w_i<w_j<z_i<z_j\},\\
\cross(\mathscr V)&:=&|\Cross(\mathscr V)|.\end{eqnarray*}
For $\mathscr V=\{(w_1,z_1),\ldots,(w_{n},z_{n})\}\in\mathscr P_2(2n)$, pairs $(w_i,z_i)$ and $(w_j,z_j)$ are said to \emph{nest} if $w_i<w_j<z_j<z_i$. The corresponding \emph{nesting} is encoded by $(w_i,w_j,z_j,z_i)$ with
\begin{eqnarray*}\Nest(\mathscr V)&:=&\{(w_i,w_j,z_j,z_i)\mid (w_i,z_i),(w_j,z_j)\in\mathscr V\text{ with } w_i<w_j<z_j<z_i\},\\
\nest(\mathscr V)&:=&|\Nest(\mathscr V)|.\end{eqnarray*}
The two concepts are illustrated in Figures~\ref{crnest1} and \ref{crnest2}, by visualizing the pair partitions as collections of disjoint chords with end-points labeled (increasing from left to right) by elements in $[2n]$. 
\label{defCrossNest}
\end{defn}

\begin{figure}\centering\includegraphics[scale=0.5]{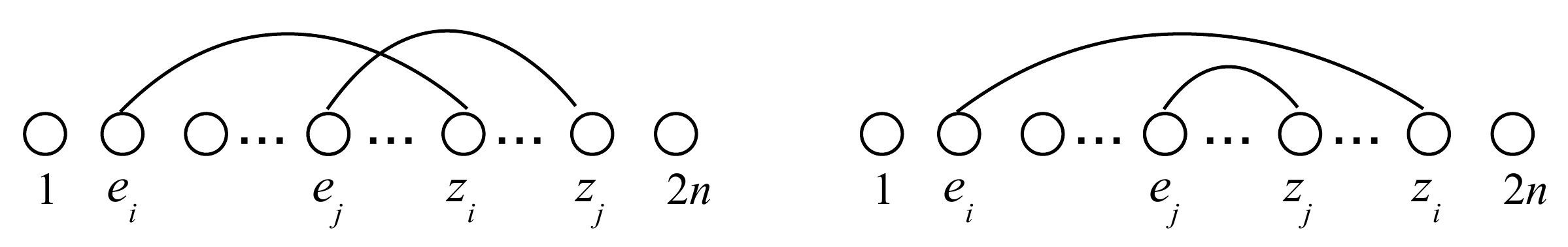}
\caption{An example of a crossing [left] and nesting [right] of a pair partition $\mathscr V=\{(e_1,z_1),\ldots,(e_{n},z_{n})\}$ of $[2n]$.}
\label{crnest1}
\end{figure}
\begin{figure}\centering\includegraphics[scale=0.5]{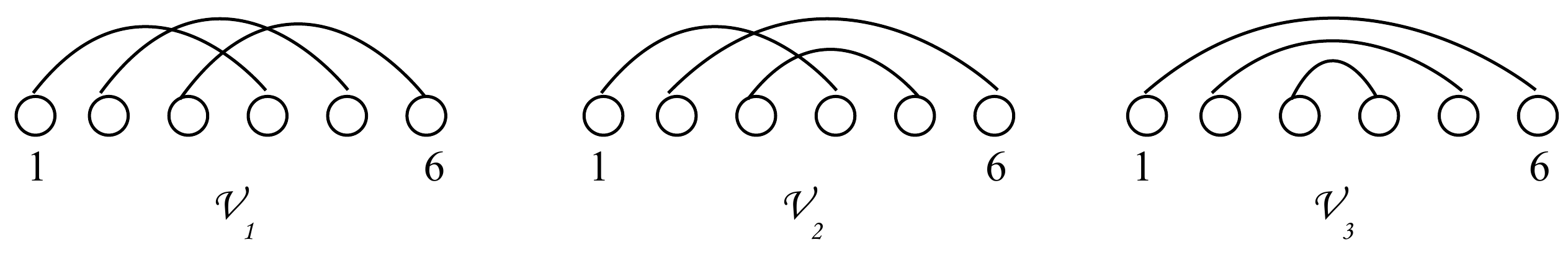}
\caption{Example of three pair paritions on $[2n]=\{1,\ldots,6\}$: $\cross(\mathscr V_1)=3,\,\nest(\mathscr V_1)=0$ [left], $\cross(\mathscr V_2)=2,\,\nest(\mathscr V_2)=1$ [middle], $\cross(\mathscr V_3)=0,\,\nest(\mathscr V_3)=3$ [right].}
\label{crnest2}
\end{figure}

Let $[n]_{q,t}$ denote the $(q,t)$-analogue of a positive integer $n$, given by
\begin{equation}[n]_{q,t} := t^{n-1}+qt^{n-2}+\ldots+q^{n-1}=(t^n-q^n)/(t-q).\end{equation}
Note that letting $t=1$ yields the usual $q$-analogue of integers, i.e. $[n]_{q}=n_{q,1}$. 
In this notation, both the generating functions of crossings in $\mathscr P_2(2n)$ and the joint generating function of crossings and nestings in $\mathscr P_2(2n)$ admit elegant continued fractions, given by
{ \begin{equation*}
\sum_{\substack{n\in\mathbb N,\\\mathscr V\in \mathscr P_2(2n)}}  \!\!\!q^{\cross(\mathscr V)}z^n
=\,\cfrac{1}{1-\cfrac{[1]_{q}z}{1-\cfrac{[2]_{q}z}{1-\cfrac{[3]_{q}z}{\ldots}}}}, \quad\quad\quad\sum_{\substack{n\in\mathbb N,\\\mathscr V\in \mathscr P_2(2n)}}  \!\!\!q^{\cross(\mathscr V)}\,t^{\nest(\mathscr V)}z^n
=\,\cfrac{1}{1-\cfrac{[1]_{q,t}z}{1-\cfrac{[2]_{q,t}z}{1-\cfrac{[3]_{q,t}z}{\ldots}}}}.\label{cfracp}
\end{equation*}}
The above continued fractions can be obtained via a classic encoding of weighted Dyck paths (see \cite{Flajolet1980}, also in a more relevant context \cite{Biane1997} and \cite{Kasraoui2006}).

The generating function of crossings in $\mathscr P_2(2n)$ admits an interesting explicit (even if not closed-form) expression, namely
$$\sum_{\mathscr V\in \mathscr P_2(2n)}  \!\!\!q^{\cross(\mathscr V)}=\frac{1}{(1-q)^n}\sum_{k=-n}^n(-1)^kq^{k(k-1)/2}{{2n}\choose{n+k}},$$
known as the \emph{Touchard-Riordan formula} \cite{Touchard1952,Touchard1950,Touchard1950-2,Riordan1975}. No analogue of the above expression is presently known for the joint generating function of crossings and nestings.

\subsection{Permutations}
Let $S_n$ denote the permutation group on $n$ letters.
.
\begin{defn}[Inversions and Coinversions] Given $\sigma\in S_n$, for $n\geq 2$, the pair $i,j\in[n]$ with $i<j$ is an \emph{inversion in $\sigma$} if $\sigma(i)>\sigma(j)$. The corresponding inversion is encoded by $(i,j)$ and the set of inversions of $\sigma$ is given by $\Inv(\sigma):=\{(i,j)\mid i,j\in[n], i<j, \sigma(i)>\sigma(j)\}$ with cardinality $\inv(\sigma):=|\Inv(\sigma)|$. Analogously, the pair $i,j\in[n]$ with $i<j$ is a \emph{coinversion in $\sigma$} if $\sigma(i)<\sigma(j)$. The corresponding coinversion is encoded by $(i,j)$ and contained in the set $\Div(\sigma):=\{(i,j)\mid i,j\in[n], i<j, \sigma(i)<\sigma(j)\}$ with cardinality $\div(\sigma):=|\Div(\sigma)|$. For $n=1$, the sets of inversions and coinversions are taken to be empty.
\label{defInv}
\end{defn}

It is well known that the generating function of the permutation inversions is given by the so-called $q$-factorial, namely
\begin{equation}
\sum_{\sigma\in S_n} q^{\inv(\sigma)}=\prod_{i=1}^n [i]_q.
\end{equation}
The above expression is in fact readily obtained as a product of the generating functions of the crossings incurred by the ``chord'' $i\mapsto\sigma(i)$ from the chords $j\mapsto \sigma(j)$ for $j>i$. By adapting this reasoning to coinversions, or by realizing that $\div(\sigma)={n\choose 2}-\inv(\sigma)$, the reader may verify that 
\begin{equation}
\sum_{\sigma\in S_n} q^{\inv(\sigma)}t^{\div(\sigma)}=\prod_{i=1}^n [i]_{q,t}.
\end{equation}

For $n\geq 2$ and $\sigma\in S_n$, it is convenient to visually represent $\sigma$ in a two-line notation. Then, $\sigma$ corresponds to a bipartite perfect matching and the inversions correspond to crossings in the diagram and the coinversions to ``non-crossings''. For instance, in adopting this representation, it becomes clear that a permutation and its inverse have the same number of inversions (and therefore also of coinversions). 

\begin{remark}
Representing a permutation $\sigma\in S_n$ in a two-line notation, aligning the two rows, and relabeling yields a unique pair-partition $\pi\in \mathscr P_2(2n)$. By additionally reversing the order of the bottom line, as illustrated in Figure~\ref{figPermPart}, the inversions will correspond to crossings of the pair partition and the co-inversions to the nestings. Indeed, given an inversion $i<j$, $\sigma(i)>\sigma(j)$, this transformation yields the pairs $(i,2n+1-\sigma(i))$ and $(j,2n+1-\sigma(j))$. Since $i<j<2n+1-\sigma(i)<2n+1-\sigma(j)$, the two pairs now form a crossing in $\pi$. Similarly, given a co-inversion $k<m$, $\sigma(k)<\sigma(m)$, the transformation yields the pairs $(k,2n+1-\sigma(k))$ and $(m,2n+1-\sigma(m))$, resulting in the ordering $k<m<2n+1-\sigma(m)<2n+1-\sigma(k)$ and yielding an inversion in $\pi$.

Naturally, this transformation of a permutation into a pair partition is by no means surjective, and is illustrative of the reasons why the set partitions (with the crossings and nestings) and the permutations (with inversions and coinversions) will both be found to feature in the algebraic structure of the $(q,t)$-Fock space.
\end{remark}

\begin{figure}\centering\includegraphics[scale=0.6]{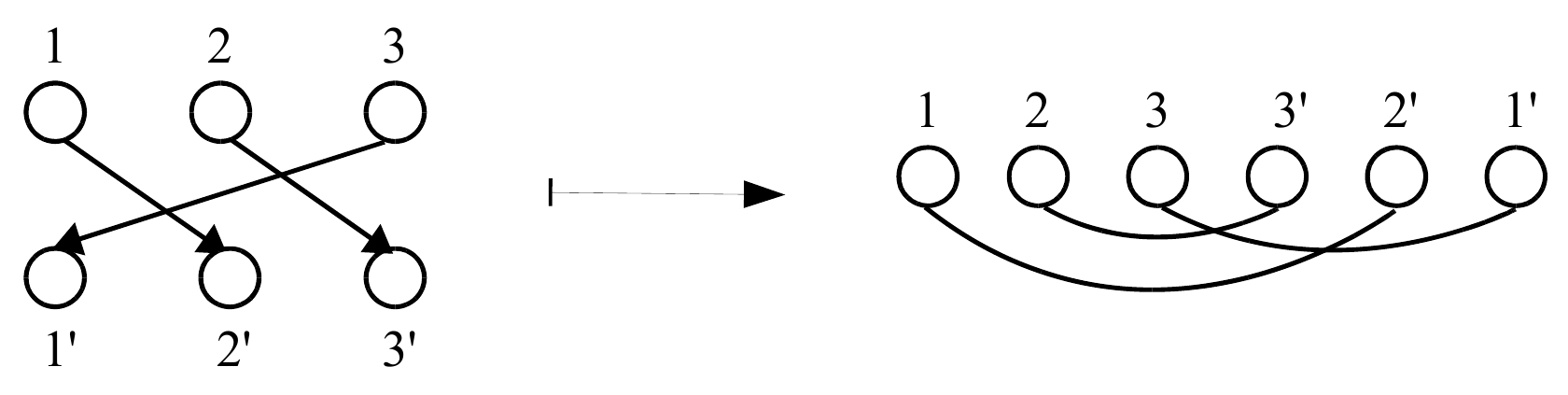}
\caption{Example of transforming the permutation $\sigma=(231)$ into a pair-partition, with permutation inversions corresponding to crossings and coinversions to nestings.}
\label{figPermPart}
\end{figure}

\subsection{Paths}

In the present context, \emph{paths} will refer to finite sequences of coordinates in the lattice $\mathbb Z\times \mathbb Z$. A North-East/South-East (NE/SE) path will refer to a path $(x_0,y_0),(x_1,y_1),\ldots,(x_n,y_n)$ where $(x_0,y_0)=(0,0)$, $x_{i}=x_{i-1}+1$ and $y_{i+1}\in\{y_{i+1}+1,y_{i+1}-1\}$ for $i\in [n]$. Indeed, interpreting the coordinates as vertices and introducing an edge between $(x_{i-1},y_{i-1})$ and $(x_i,y_i)$ for $i\in [n]$ yields a trajectory of a walker in the plane, starting out at the origin and moving, at each step of length $\sqrt 2$, either in the NE or SE direction. A \emph{Dyck path of length $2n$} is NE/SE path $(x_0,y_0),(x_1,x_2),\ldots,(x_n,y_n)$ where $y_i\geq 0$ for all $i\in [n-1]$ and $y_n=0$. Given $n\in\mathbb N$, the set of Dyck paths of length $2n$ will be denoted by $D_n$. The reader is referred to Figure~\ref{pathDyck} of Section~\ref{secqtSemicircular} for an illustration. 

Counted by the Catalan number $C_n=(n+1)^{-1} {2n\choose n}$, Dyck paths are found to be in bijective correspondence with a surprising number of combinatorial objects. In the present context, in Section~\ref{secqtSemicircular}, weighted Dyck paths will be found to encode the moments of the creation and field operators on the $(q,t)$-Fock space.

\section{The $(q,t)$-Fock Space}
\label{secqtFock}
The present section constructs the $(q,t)$-Fock space as a refinement of the construction introduced in \cite{Bozejko1991}.

Consider a real, separable Hilbert space $\mathscr H$ and some distinguished vector $\Omega$ disjoint from $\H$. Let $\mathscr F=(\mathbb C \Omega)\oplus\bigoplus_{n\geq 1} \H_{\mathbb C}^{\otimes n}$, where $\H_{\mathbb C}$ is the complexification of $\H$ and both the direct sum and tensor product are understood to be algebraic. In particular, $\mathscr F$ can be viewed as the vector space over $\mathbb C$ generated by $\{\Omega\}\cup\{h_1\otimes\ldots\otimes h_n\}_{h_i\in \mathscr H,n\in\mathbb N}$.

\sp
For $f\in \H$ and $q,t\in\R$, define the operators $a(f)^\ast $ and $a(f)$ on $\F$ by linear extension of:
\begin{equation}a(f)^\ast \Omega = f,\quad\quad a(f)^\ast h_1\otimes\ldots\otimes h_n=f\otimes h_1\otimes\ldots\otimes h_n\label{cdef1}\end{equation}
and
\begin{equation}a(f)\Omega = 0, \quad\quad a(f)h_1\otimes\ldots\otimes h_n=\sum_{k=1}^n q^{k-1}t^{n-k} \langle f,h_k\rangle_{_\H}\,  h_1\otimes \ldots\otimes\breve h_k\otimes\ldots\otimes h_n,\label{cdef2}\end{equation}
where the superscript $\breve h_k$ indicates that $h_k$ has been deleted from the product. Note that the creation operator $a(f)^\ast $ is defined identically to the operator $c^\ast(f)$ in \cite{Bozejko1991}, whereas the ``twisted'' annihilation operator $a(f)$ is the refinement of the operator $c(f)$ in \cite{Bozejko1991} by a second parameter, $t$. Next, define the operator $t^N$ on $\F$ by linear extension of
\begin{equation}t^N \Omega = \Omega,\quad\quad t^N h_1\otimes\ldots\otimes h_n=t^n\,h_1\otimes\ldots\otimes h_n.\end{equation}
For $t>0$, the operator can be written in a somewhat more natural form as $e^{\alpha N}$, where $\alpha=\log(t)$.

\begin{lem} For all $f,g\in\H$, the operators $a(f),a(g)^\ast$ on $\F$ fulfill the relation
$$a(f)a(g)^\ast -q \,a(g)^\ast a(f)= \langle f,g\rangle_{_\H}\, t^{N}.$$
\label{comrel}
\end{lem}
\begin{proof} For any $n\in\mathbb N$ and $g,h_1,\ldots,h_n\in\H$,
\begin{align*}&a(f)a(g)^\ast h_1\otimes\ldots\otimes h_n= a(f) g\otimes h_1\otimes\ldots\otimes h_n\\
&= t^{n} \langle f,g\rangle_{_\H}\, h_1\otimes\ldots\otimes h_n + \sum_{k=2}^{n+1} q^{k-1}t^{n+1-k}\langle f,h_{k-1}\rangle_{_\H} g\otimes h_1\otimes \ldots\otimes\breve h_{k-1}\otimes\ldots\otimes h_n\\
&= t^{n} \langle f,g\rangle_{_\H}\, h_1\otimes\ldots\otimes h_n + g\otimes\left(\sum_{k=1}^{n} q^{k}t^{n-k}\langle f,h_{k}\rangle_{_\H} h_1\otimes \ldots\otimes\breve h_{k}\otimes\ldots\otimes h_n\right)\\
&= t^{n} \langle f,g\rangle_{_\H}\, h_1\otimes\ldots\otimes h_n +  q\, a(g)^\ast a(f) h_1\otimes\ldots\otimes h_n
\end{align*}
\end{proof}

Define the sesquilinear form $\langle,\rangle_{q,t}$ on $\F$ by
\begin{equation}\langle g_1\otimes\ldots\otimes g_n,h_1\otimes\ldots\otimes h_m\rangle_{q,t}=0\quad \text{for } m\neq n\label{defform1}\end{equation}
and otherwise recursively by
\begin{equation}\langle g_1\otimes\ldots\otimes g_n,h_1\otimes\ldots\otimes h_n\rangle_{q,t}=\sum_{k=1}^n q\,^{k-1}\,\,t\,^{n-k}\langle g_1,h_k\rangle_{_\H}\langle g_2\otimes\ldots\otimes g_n,h_1\otimes \ldots\otimes\breve h_{k}\otimes\ldots\otimes h_n\rangle_{q,t}.\label{defform}\end{equation}

\begin{remark} For all $q$, setting $t=1$ recovers the usual inner product on the $q$-Fock space of \cite{Bozejko1991}. Letting $t=0$ yields a sesquilinear form on the full Fock space that is given by
$$\langle g_1\otimes\ldots\otimes g_n,h_1\otimes\ldots\otimes h_n\rangle_{q,0}=q^{n\choose 2}\langle g_1,h_n\rangle_{_\H}\ldots \langle g_n,h_1\rangle_{_\H},$$
which does not generally satisfy the positivity requirement of an inner product.
\end{remark}

The range of $t$ required for $\langle,\rangle_{q,t}$ to be an inner product will be characterized shortly, in Lemma~\ref{lempositivity}. But, first, note that for all $f\in\H$, $a(f)^\ast$ is indeed the adjoint of $a(f)$ with respect to $\langle,\rangle_{q,t}$.

\begin{lem} For all $f\in\H$, $\xi,\eta\in \F$,
$$\langle a(f)^\ast \xi,\eta\rangle_{q,t}=\langle\xi,a(f)\eta\rangle_{q,t}.$$\end{lem}
\begin{proof} It suffices to note that, directly from the previous definitions,
\begin{align*}&\langle a(f)^\ast g_1\otimes\ldots\otimes g_n,h_1\otimes\ldots\otimes h_{n+1}\rangle_{q,t}=\langle f\otimes g_1\otimes\ldots\otimes g_n,h_1\otimes\ldots\otimes h_{n+1}\rangle_{q,t}\\
&=\sum_{k=1}^{n+1}q\,^{k-1}\,\, t\,^{n+1-k}  \langle f,h_k\rangle_{_\H}\, \langle g_1\otimes\ldots\otimes g_n,h_1\otimes \ldots\otimes\breve h_{k}\otimes\ldots\otimes h_{n+1}\rangle_{q,t}\\&=\langle g_1\otimes\ldots\otimes g_n,a(f)h_1\otimes\ldots\otimes h_{n+1}\rangle_{q,t}
\end{align*}
\end{proof}

Still in line with \cite{Bozejko1991}, it is convenient define the ``projection" operator $P_{q,t}:\F\to\F$ allowing one to express the sesquilinear form $\langle,\rangle_{q,t}$ via the usual scalar product $\langle,\rangle_0$ on the full Fock space. Consider the unitary representation $\pi\mapsto U_\pi^{(n)}$ of the symmetric group $S_n$ on $\H^{\otimes n}$, given by
$$U_\pi^{(n)} \,h_1\otimes\ldots\otimes h_n=h_{\pi(1)}\otimes\ldots\otimes h_{\pi(n)}.$$
Recalling the permutation statistics given by inversions and co-inversions, defined in the previous section (cf. Definition~\ref{defInv}), let
$$P_{q,t}=\bigoplus_{n=0}^\infty P_{q,t}^{(n)}\quad\quad\text{with}\quad P_{q,t}^{(n)}:\H^{\otimes n}\to \H^{\otimes n},$$
\begin{equation}P_{q,t}^{(n)}:=\sum_{\pi\in S_n} q^{\inv(\pi)}t^{\div(\pi)}\,U_\pi^{(n)},\end{equation}
where for the unique $\pi\in S_1$, it is understood that $\inv(\pi)=\div(\pi)=0$. Thus, $P_{q,t}^{(1)}=1$ and there is no change to the scalar product on the single-particle space. For $n\geq 2$, note that for every $1\leq i<j\leq n$, the pair $(i,j,\pi(i),\pi(j))$ is either an inversion or a coinversion, and therefore $\inv(\pi)+\div(\pi)={n\choose 2}$. It follows that for $n\geq 2$ and $t\neq 0$, the projection\footnote{Strictly speaking, $P_{q,t}$ is not a projection for general values of $q$ and $t$, as $\left(P_{q,t}^{(n)}\right)^2=\sum_{\pi\in S_n} q^{2\inv(\pi)}t^{2\div(\pi)}\,U_\pi^{(n)}$, which follows from the fact that a permutation and its inverse share the same number of (co)inversions (cf. Section~\ref{preliminaries}). The terminology is inherited from the classical construction, as $P_{q,t}^2=P_{q,t}$ if (and only if) $|q|=t=1$.} operator can equivalently be expressed as
\begin{equation}P_{q,t}^{(n)}=t^{n\choose 2}\sum_{\pi\in S_n} \left(\frac{q}{t}\right)^{\inv(\pi)}\,U_\pi^{(n)}=t^{n\choose 2}P_{q/t}^{(n)},\label{alternativeP}\end{equation}
where $P_{q}^{(n)}$ denotes the projection operator on the subspace $\H^{\otimes n}$ of the $q$-Fock space (see \cite{Bozejko1991}).

\begin{lem} For all $\xi, \eta\in\F$,
$$\langle\eta,\xi\rangle_{q,t}=\langle\xi,P_{q,t}\,\eta\rangle_0.$$\label{lemP}
\end{lem}
\begin{proof} It suffices to prove that for all $n\in\mathbb N$ and all $g_i,h_j\in\H$,
\begin{eqnarray*}\langle g_1\otimes\ldots\otimes g_n,h_1\otimes\ldots\otimes h_n\rangle_{q,t}&=&\langle g_1\otimes\ldots\otimes g_n,P_{q,t}h_1\otimes\ldots\otimes h_n\rangle_0\\
&=&\sum_{\pi\in S_n}q^{\inv(\pi)}t^{\div(\pi)}\langle g_1,h_{\pi(1)}\rangle_\H\ldots\langle g_n,h_{\pi(n)}\rangle_\H
\end{eqnarray*}
The claim clearly holds for $n=1$. Proceding by induction on $n$, recall that from the definition of $\langle,\rangle_{q,t}$,
\begin{eqnarray*}\langle g_1\otimes\ldots\otimes g_n,h_1\otimes\ldots\otimes h_n\rangle_{q,t}&=&\sum_{k=1}^n q^{k-1}t^{n-k}\langle g_1,h_k\rangle_{_\H}\langle g_2\otimes\ldots \otimes g_n,h_1\otimes\ldots\otimes \breve h_k\otimes \ldots\otimes h_n\rangle_{q,t}.\end{eqnarray*}
Letting $S_n^{(k)}$ denote all bijections from $\{2,\ldots,n\}$ to  $\{1,\ldots,\breve k,\ldots, n\}$, note that, for all $n\in\mathbb N$, every $\pi\in S_n$ can be uniquely decomposed as a pair $(k,\sigma)$ for some $k\in [n]$ and $\sigma\in S_{n-1}^{(k)}$ and, conversely, that any such pair gives a distinct element of $S_n$. Specifically, let $\pi(1)=k$ and $\pi(\ell)=\sigma(\ell)$ for $\ell \in \{2,\ldots,n\}$. Then, noting that the natural correspondences $[n-1]\leftrightarrow \{2,\ldots,n\}$ and $[n-1]\leftrightarrow \{1,\ldots,\breve k,\ldots n\}$ are order-preserving, the inductive hypothesis on $n-1$ can be written as
$$\langle g_2\otimes\ldots \otimes g_n,h_1\otimes\ldots\otimes \breve h_k\otimes \ldots\otimes h_n\rangle_{q,t}=\sum_{\sigma\in S_{n-1}^{(k)}}q^{\inv(\sigma)}t^{\div(\sigma)}\langle g_2,h_{\sigma(2)}\rangle_\H\ldots \langle g_n,h_{\sigma(n)}\rangle_\H,$$
where $\inv(\sigma)$ counts all the pairs $(i,j)\in \{2,\ldots,n\}\times \{1,\ldots,\breve k,\ldots, n\}$ with $i<j$ and $\pi(i)>\pi(j)$  and $\div(\sigma)$ is defined analogously. Furthermore, observing that $\inv(\pi) = \inv(\sigma) + k-1$ and $\div(\pi)=\div(\sigma)+n-k$, as demonstrated in the caption of Figure~\ref{proof1}, it follows that
\begin{align*}&\langle g_1\otimes\ldots\otimes g_n,h_1\otimes\ldots\otimes h_n\rangle_{q,t}\\
&= \sum_{k=1}^n q^{k-1}t^{n-k}\langle g_1,h_k\rangle_{_\H} \sum_{\sigma\in S_{n-1}^{(k)}}q^{\inv(\sigma)}t^{\div(\sigma)}\langle g_2,h_{\sigma(2)} \rangle_\H\ldots\langle g_n,h_{\sigma(n)}\rangle_\H\\
&=\sum_{\pi\in S_{n}}q^{\inv(\pi)}t^{\div(\pi)}\langle g_1,h_{\pi(1)}\rangle_\H\langle g_2,h_{\pi(2)}\rangle_\H\ldots\langle g_n,h_{\pi(n)}\rangle_\H\end{align*}
\end{proof}

\begin{figure} \centering
\includegraphics[scale=0.5]{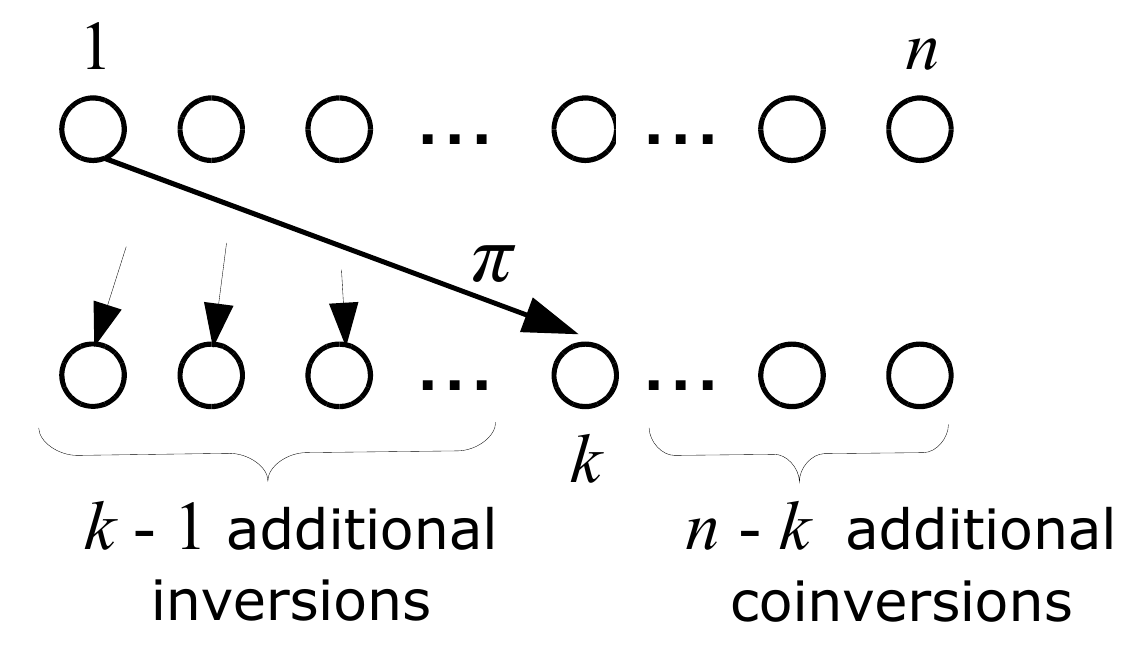}
\caption{\small If $\pi(1)=k$, then there are exactly $k-1$ elements from the set $\{2,\ldots,n\}$ that map under $\pi$ to an element in $\{1,\ldots,k-1\}$ and exactly $n-k$ elements from $\{2,\ldots,n\}$ that map to $\{k+1,\ldots,n\}$. Thus, there are $k-1$ elements $j>1$ for which $\pi(j)<\pi(1)$ and $n-k$ elements $\ell>1$ for which $\pi(\ell)>\pi(1)$. For any other pair $(i,j)\in [n]^2$ with $i\neq 1$, the corresponding inversion or coinversion is given by the map $\sigma:\{2,\ldots,n\}\to\{1,\ldots,\breve k,\ldots,n\}$. It follows that $\inv(\pi)=k-1+\inv(\sigma)$ and $\div(\pi)=n-k+\div(\sigma)$.}\label{proof1}
\end{figure}

While the following facts, contained in Lemmas~\ref{lempositivity} through \ref{lemmanorm}, have direct proofs analogous to those in \cite{Bozejko1991}, it is more convenient to use (\ref{alternativeP}) and derive the desired properties from those of the $q$-Fock space.

\begin{lem}
a) The operator $P_{q,t}$ is positive for all $|q|\leq t$.\\
b) The operator $P_{q,t}$ is strictly positive for all $|q|< t$.
\label{lempositivity}
\end{lem}
\begin{proof} Since $P_{q,t}=\bigoplus_{n=0}^\infty P_{q,t}^{(n)}$, it suffices to consider the positivity of $P_{q,t}^{(n)}$. Since $t^{n\choose 2} >0$, the positivity of $P_{q,t}^{(n)}$ for $|q|\leq t$ follows from (\ref{alternativeP}) by the positivity (\cite{Bozejko1991}) of $P_{q/t}^{(n)}$ for $|q/t|\leq 1$ and the strict positivity of of $P_{q,t}^{(n)}$ for $|q|\leq t$ follows from the strict positivity (\cite{Bozejko1991}) of $P_{q/t}^{(n)}$ for $|q/t|< 1$.
\end{proof}

\noindent Since $\langle\,,\,\rangle_{q,t}$ is an inner product on $\F$, the completion of $\F$ yields the desired $(q,t)$-Fock space.

\begin{defn} The $(q,t)$-Fock space $\mathscr F_{q,t}$ is the completion of $\F$ with respect to $\langle\,,\,\rangle_{q,t}$.
\end{defn}

\begin{remark} For $\text{dim}(\H)\geq 2$, the conditions of Lemma~\ref{lempositivity} are also necessary. In particular, let $e_1,e_2$ be two unit vectors in $\H$ with $\langle e_1,e_2\rangle_{q,t}=0$. Note that
$$\|e_1\otimes e_2+e_2\otimes e_1\|_{q,t}=2t+2q,\quad\quad\text{and}\quad\quad \| e_1\otimes e_2-e_2\otimes e_1\|_{q,t}=2t-2q.$$
Thus, $\langle\,,\,\rangle_{q,t}$ is positive (resp. strictly positive) only if $|q|\leq t$ (resp. $|q|< t$).
\end{remark}

Analogously the case of the $q$-Fock space of \cite{Bozejko1991}, the operators $a(f)$ on $\F_{q,t}$ are bounded for $0\leq -q \leq t\leq 1$ and $0<q<t\leq 1$. Letting $q=1$ (and therefore $t=1$) recovers the Bosonic Fock space, in which case $a(f)^\ast$ and $a(f)$ are unbounded and defined only on the dense domain $\F$. 

\begin{lem}
For any $f\in\H$, the operator $a(f)$ on $\F$ is bounded for $0\leq |q|<t\leq 1$, with norm given by
$$\displaystyle\|a(f)\|=\left\{\begin{array}{ll}\|f\|_\H& 0\leq -q\leq t\leq 1,\\&\\\frac{1}{\sqrt{1-q}}\,\|f\|_\H&0<q<t=1\\&\\\sqrt{\frac{t^{n_\ast}-q^{n_\ast}}{t-q}}\,\|f\|_\H&0<q<t<1\end{array}\right.,$$
for
$$n_\ast=\left\lceil\frac{\log\left(1-q\right)-\log\left(1-t\right)}{\log\left(t\right)-\log\left(q\right)}\right\rceil.$$
\label{lemmanorm}
\end{lem}
\begin{proof}
Let $q\in [-1,0]$ and $|q|\leq t\leq 1$. For $\xi_n\in\H^{\otimes n}$, by Lemma~\ref{comrel},
\begin{eqnarray*}\langle a(f)^\ast \xi_n,a(f)^\ast \xi_n\rangle_{q,t}&=&\langle\xi_n,a(f)a(f)^\ast \xi_n\rangle_{q,t}=t^{n}\,\langle f,f\rangle_\H\,\langle \xi_n,\xi_n\rangle_{q,t}+q\langle \xi_n,a(f)^\ast a(f)\xi_n\rangle_{q,t}\\
&=& t^{n}\,\|f\|_\H^2\,\|\xi_n\|^2_{q,t}+q\|a(f)\xi_n\|^2_{q,t}\leq \|f\|_\H^2\,\|\xi_n\|^2_{q,t},\end{eqnarray*}
as $q\leq 0$ and $0<t\leq 1$.  Next, for an element $\xi=\sum_{i=0}^n \alpha_i\,\xi_i\in\mathscr F$, where $\alpha_i\in\mathbb C$ and $\xi_i\in\H^{\otimes i}$, that $a(f)^\ast$ is linear and $\langle\xi_n,\xi_m\rangle_{q,t}=0$ whenever $n\neq m$ implies that
$$\langle a(f)^\ast \xi,a(f)^\ast \xi\rangle_{q,t}=\sum_{i=0}^n|\alpha_i|^2\langle a(f)^\ast \xi_i,a(f)^\ast \xi_i\rangle_{q,t}\leq \sum_{i=0}^n|\alpha_i|^2 \|f\|_\H^2\,\|\xi_i\|^2_{q,t}=\|f\|_\H^2\|\xi\|^2_{q,t}.$$
Finally, since $\langle a(f)^\ast \Omega,a(f)^\ast \Omega\rangle_{q,t}=\|f\|_\H^2$, it follows that $\|a(f)\|=\|a(f)^\ast \|=\|f\|_\H$.

For $q\in (0,1)$, analogously to the previous case, it suffices to focus on $\xi\in\H^{\otimes n}$. Again, $\langle a(f)^\ast \Omega,a(f)^\ast \Omega\rangle_{q,t}=\|f\|_\H^2$ and, for $n\in\mathbb N$, (\ref{alternativeP}) yields
\begin{eqnarray*}
\langle a(f)^\ast \xi,a(f)^\ast \xi\rangle_{q,t}&=&\langle a(f)^\ast \xi,P_{q,t}^{(n+1)}a(f)^\ast \xi\rangle_0=t^{n+1\choose 2}\langle a(f)^\ast \xi,P_{q/t}^{(n+1)}a(f)^\ast \xi\rangle_0\\&=&t^{n+1\choose 2}\langle a(f)^\ast \xi,a(f)^\ast \xi\rangle_{q/t}.
\end{eqnarray*}
Recalling that the creation operators on the $(q,t)$-Fock and $\mu$-Fock spaces (for $\mu\in[-1,1]$) are defined by the linear extension of the same operator on the dense linear subspace $\mathscr F$, it follows that
$$\langle a(f)^\ast \xi,a(f)^\ast \xi\rangle_{q/t}=\langle a_{q/t}(f)^\ast\xi,a_{q/t}(f)^\ast\xi\rangle_{q/t},$$
where $a_{q/t}^\ast(f)$ analogously denotes the creation operator on the $(q/t)$-Fock space. Recalling, furthermore, the familiar bounds on $\|a_q^\ast(f)\xi\|_{q}$ (e.g. Lemma 4 in \cite{Bozejko1991}),
$$\langle a(f)^\ast \xi,a(f)^\ast \xi\rangle_{q,t}\leq t^{n+1\choose 2}\frac{1-(q/t)^{n+1}}{1-q/t}\|f\|_{\H}^2\|\xi\|_{q/t}^2=t^{n+1\choose 2}\frac{1-(q/t)^{n+1}}{1-q/t}\|f\|_{\H}^2\|\xi\|_{q/t}^2.$$
But,
$$\|\xi\|_{q/t}^2=\langle \xi,P_{q/t}^{(n)}\xi\rangle_0=t^{-{n\choose 2}}\langle \xi,P_{q,t}^{(n)}\xi\rangle_0=t^{-{n\choose 2}}\|\xi\|_{q,t}^2,$$
and so,
$$\langle a(f)^\ast \xi,a(f)^\ast \xi\rangle_{q,t}\leq t^{n}\frac{1-(q/t)^{n+1}}{1-q/t}\|f\|_{\H}^2\|\xi\|_{q,t}^2=\frac{t^{n+1}-q^{n+1}}{t-q}\|f\|_{\H}^2\|\xi\|_{q,t}^2.$$
Since $0<t\leq 1$, it follows that $a(f)^\ast $ is bounded. To recover the corresponding expression for the norm, let $\xi=f\otimes \ldots\otimes f=f^{\otimes n}\in\H^{\otimes n}$ and therefore $a(f)^\ast \xi=f^{\otimes n+1}$. Thus,
\begin{eqnarray*}\langle a(f)^\ast \xi,a(f)^\ast \xi\rangle_{q,t}=\langle f^{\otimes n+1},f^{\otimes n+1}\rangle_{q,t}=\sum_{k=1}^{n+1}q^{k-1}t^{n+1-k}\|f\|_\H^2\|f^{\otimes n}\|_{q,t}^2=\frac{t^{n+1}-q^{n+1}}{t-q}\|f\|_\H^2\|\xi\|_{q,t}^2\end{eqnarray*}
For $t=1$,
$$\sup_{n\in\mathbb N}\, \frac{\langle a(f)^\ast f^{\otimes n},a(f)^\ast f^{\otimes n}\rangle_{q,t}}{\|f^{\otimes n}\|_{q,t}^2}=\frac{1}{1-q}\|f\|_\H^2.$$
Otherwise, it remains to compute $n\in\mathbb N$ that maximizes $t^{n+1}-q^{n+1}$. Let
$$r_n:=\frac{t^{n-1}-t^{n}}{q^{n-1}-q^{n}},\quad\quad n\in\mathbb N\cup\{0\},$$
and note that, since $0<q<t$, both the numerator and the denominator are strictly positive. Furthermore, $\{r_n\}_{n\geq 0}$ forms a strictly increasing sequence as 
$$r_n=\frac{t^n}{q^n}\frac{\left(1-t\right)}{\left(1-q\right)}=\frac{t}{q}\,r_{n-1}>r_{n-1}.$$ Now, note that $r_n> 1$ iff $t^n-q^n< t^{n-1}-q^{n-1}$ and, conversely, $r_n\leq 1$ iff $t^n-q^n\geq t^{n-1}-q^{n-1}$. It follows that $t^n-q^n$ is maximized for $n=n_\ast$, where $n_\ast$ is the greatest non-negative integer for which $r_n\leq 1$. A straightforward calculation then yields
$$n_\ast=\left\lceil\frac{\log\left(1-q\right)-\log\left(1-t\right)}{\log\left(t\right)-\log\left(q\right)}\right\rceil.$$
\end{proof}

\section{The $(q,t)$-Gaussian Processes}
\label{secqtSemicircular}
The natural starting point to the probabilistic considerations of this section is the $\ast$-probability space $(\mathscr{G}_{q,t},\varphi_{q,t})$, discussed next, formed by the $\ast$-algebra generated by the creation and annihilation operators on $\F_{q,t}$ and the vacuum-state expectation on the algebra. The $\ast$-probability space $(\mathscr{G}_{q,t},\varphi_{q,t})$ provides a convenient setting in which the $(q,t)$-Gaussian family is subsequently introduced and studied.

\subsection{The $(\mathscr{G}_{q,t},\varphi_{q,t})$ $\ast$-probability space}
For a clear introduction to non-commutative probability spaces, the reader is referred to the monograph \cite{NicaSpeicher}. In the present context, the setting of interest is that of the $(\mathscr{G}_{q,t},\varphi_{q,t})$ $\ast$-probability space, formed by:
\begin{itemize}
\item the unital $\ast$-algebra $\mathscr{G}_{q,t}$ generated by $\{a(h)\mid h\in \H\}$;
\item the unital linear functional $\varphi_{q,t}:\mathscr{G}_{q,t}\to\mathbb C$, $b\mapsto \langle \Omega,b\Omega \rangle_{q,t}$ (i.e. the vacuum expectation state on $\mathscr{G}_{q,t}$.).
\end{itemize}
In this non-commutative setting, \emph{random variables} are understood to be the elements of $\mathscr{G}_{q,t}$ and of particular interest are their joint mixed moments, i.e. expressions $\varphi_{q,t}(b_1^{\epsilon(1)}\ldots b_k^{\epsilon(k)})$ for $b_1,\ldots,b_k\in \mathscr{G}_{q,t}$, $k\in\mathbb N$ and $\epsilon(1),\ldots,\epsilon(k)\in\{1,\ast\}$. 
Since $\varphi_{q,t}$ is linear, one is allowed to focus on the joint moments $\varphi_{q,t}(a(h_1)^{\epsilon(1)}\ldots a(h_k)^{\epsilon(k)})$. These turn out to be intimately connected with two combinatorial objects, discussed next.

For $n\in\mathbb N$, consider first the map $\psi_n$ from $\{1,\ast\}^{n}$ to the set of all NE/SE paths of length $n$ (cf. Section~\ref{preliminaries}), by which every ``$\ast$'' maps to a NE step and every ``$1$'' maps to a SE step. Clearly, $\psi_n$ is a bijection and of particular interest is the set $\psi_{2n}^{-1}(D_n)$, where $D_n$ denotes the set of Dyck paths of length $2n$. An example of an element in $\{1,\ast\}^{14}$ mapping into $D_7$ is shown in Figure~\ref{pathDyck}.

\begin{figure}[h]\centering
\includegraphics[scale=0.5]{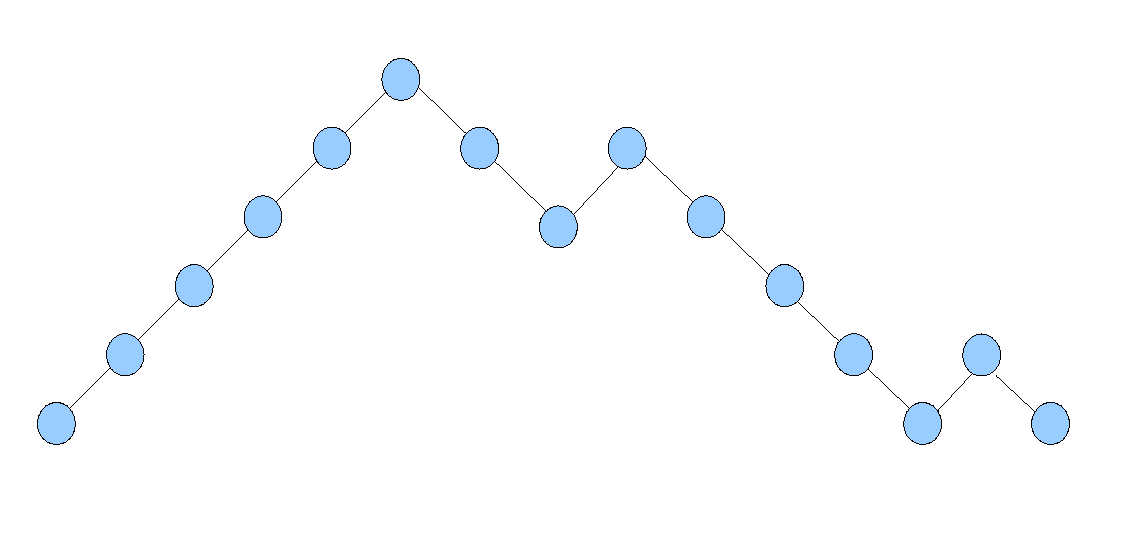}
\caption{$\psi_{14}(\ast,\ast,\ast,\ast,\ast,1,1,\ast,1,1,1,1,\ast,1)$.}\label{pathDyck}
\end{figure}

Given $(\epsilon(1),\ldots,\epsilon(2n))\in \psi^{-1}(D_n)$, let each $\ast$ encode an ``opening'' and each $1$ encode a ``closure'' and, given some such string of openings and closures, consider all the ways in which the elements of the string can be organized into disjoint pairs so that (1) an opening is always to the left of the corresponding closure and (2) no opening is left unpaired. 
Clearly, each fixed $(\epsilon(1),\ldots,\epsilon(2n))\in \psi^{-1}(D_n)$ is thus associated with a distinct set of pair-partitions of $[2n]$, that is, a subset of $\mathscr P_2(2n)$ (cf. Section~\ref{preliminaries}). Conversely, every pair partition naturally corresponds to a unique string of openings/closures belonging to $\psi^{-1}(D_n)$. Thus, writing $\mathscr V\sim_p\mathscr V'$ for any $\mathscr V,\mathscr V'\in \mathscr P_2(2n)$ for which the corresponding strings of openings and closures both encode the same string in $\psi^{-1}(D_n)$ defines an equivalence relation. Figure~\ref{pairings_example} shows the six pairings in the equivalence class of the string $(\ast,\ast,\ast,1,1,1)\in\psi^{-1}(D_6)$.

\begin{figure}[h]\centering
\includegraphics[scale=0.5]{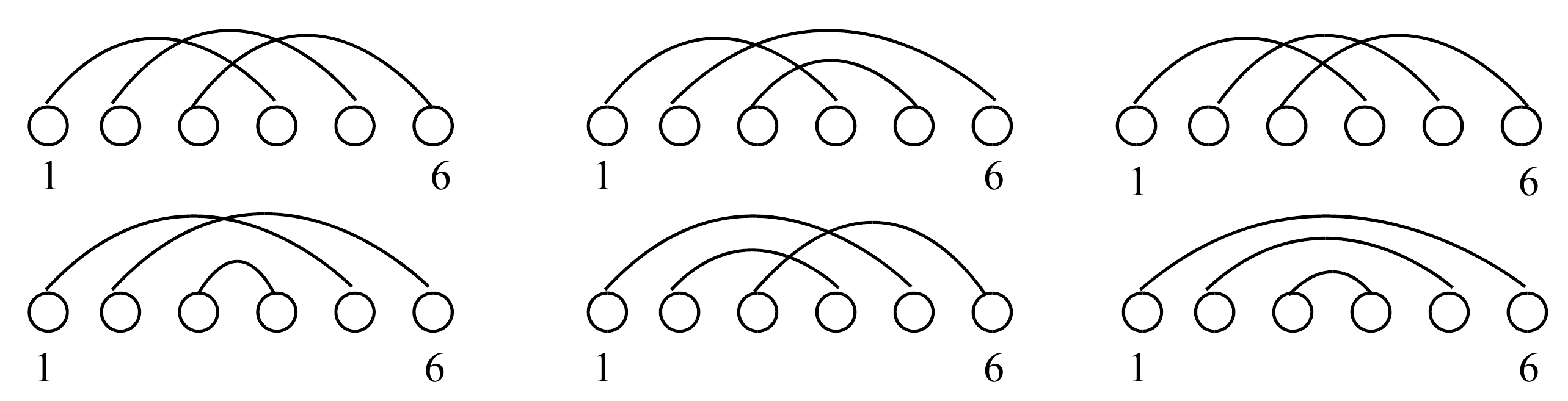}
\caption{The equivalence class of $(\ast,\ast,\ast,1,1,1)\in\psi^{-1}(D_6)$ in $P_2(6)$.}\label{pairings_example}
\end{figure}

Next, recalling that the purpose is calculating the value of the mixed moment $$\varphi_{q,t}(a(h_1)^{\epsilon(1)}\ldots a(h_{2n})^{\epsilon(2n)}),$$ fix a string $(\epsilon(1),\ldots,\epsilon(2n))\in\psi_{2n}^{-1}(D_n)$ and jointly consider some underlying choice of $h_1,\ldots,h_{2n}\in \H$. Each pairing $\mathscr V=\{(w_1,z_1),\ldots,(w_n,z_n)\}$ in the equivalence class of $(\epsilon(1),\ldots,\epsilon(2n))$ is then assigned the weight
$$\text{wt}(\mathscr V;h_1,\ldots,h_{2n}):=q^{\cross(\mathscr V)}t^{\nest(\mathscr V)}\prod_{i=1}^n\langle h_{w_i},h_{z_i}\rangle_\H,$$
where $\cross(\mathscr V)$ and $\nest(\mathscr V)$ correspond, respectively, to the numbers of crossings and nestings in $\mathscr V$, as defined in Section~\ref{preliminaries}. For example, the top left-most pairing in Figure~\ref{pairings_example} is thus given the weight $q^3\langle g_1,g_4\rangle_\H\langle g_2,g_5\rangle_\H\langle g_3,g_6\rangle_\H$. 
For $n\in\mathbb N$, let $T_{q,t}(h_1,\ldots,h_{2n};\epsilon(1),\ldots,\epsilon(2n))$ denote the generating function of the weighted pair-partitions of $[2n]$, namely
\begin{eqnarray}T_{q,t}(h_1,\ldots,h_{2n};\epsilon(1),\ldots,\epsilon(2n))&:=&\sum_{\substack{\mathscr V\in \mathscr P_{2}(2n)\\\mathscr V\sim_p (\epsilon(1),\ldots,\epsilon(2n))}}\text{wt}(\mathscr V;h_1,\ldots,h_{2n})\nonumber\\
&=& \sum_{\substack{\mathscr V\in \mathscr P_{2}(2n)\\\mathscr V\sim_p (\epsilon(1),\ldots,\epsilon(2n))}}q^{\cross(\mathscr V)}t^{\nest(\mathscr V)}\prod_{i=1}^n\langle h_{w_i},h_{z_i}\rangle_\H
\label{defT}
\end{eqnarray}
where $\mathscr V\sim_p (\epsilon(1),\ldots,\epsilon(2n))$ is meant to indicate that $\mathscr V$ is in the equivalence class of $(\epsilon(1),\ldots,\epsilon(2n))$ under $\sim_p$ (i.e. $\mathscr V$ has its opening/closure string given by $(\epsilon(1),\ldots,\epsilon(2n))\,$). Writing the weight of a pairing as a product of weights of its pairs, as in the following lemma, is the remaining ingredient in connecting the combinatorial structures at hand to the moments $\varphi_{q,t}(a(h_1)^{\epsilon(1)}\ldots a(h_{2n})^{\epsilon(2n)})$.

\begin{lem} For $\mathscr V=\{(z_1,w_1),\ldots,(z_n,w_n)\}\in\mathscr P_2(2n)$ with $w_1<w_2<\ldots<w_n$ (i.e. with pairs indexed in the increasing order of closures) and $h_1,\ldots,h_{2n}\in\H$,
$$\text{wt}(\mathscr V;h_1,\ldots,h_{2n})=\prod_{i=1}^n q^{|\{z_j\mid z_i<z_j<w_i,j>i\}|}\,\,t^{|\{z_j\mid 1\leq z_j<z_i,j>i\}|}\langle h_{z_i},h_{w_i}\rangle_\H.$$\label{lemrecursive}
\end{lem}
\begin{proof} Consider the following procedure for assigning weights to a pairing. Starting with the left-most closure $w_1$, suppose that it connects to some opening $z_1$ (where $z_1<w_1$). Since $w_1$ is indeed the left-most closure, exactly $w_1-z_1$ pairs will cross the pair $(z_1,w_1)$ and exactly $z_1-1$ pairs will nest with it. Thus, assign the pair $(z_1,w_1)$ the weight $q^{w_i-z_i}t^{z_i-1} \langle h_{z_i},h_{w_i}\rangle_\H$. Consider now the $i^\text{th}$ closure from the left, $w_i$, and let $\mathscr V'$ be the pairing from which all pairs $(z_k,w_k)$ for $k<i$ have been removed. Then, assigning the pair $(z_i,w_i)$ its weight in $\mathscr V'$ according to the previous recipe does not take into acount any crossings or nestings that have already been accounted for by the previous pairs. It now suffices to note that, by the end of the procedure, all the crossings and nestings have been taken into account and that the product of the weights of the pairs indeed equals the expression for $\text{wt}(\mathscr V;h_1,\ldots,h_{2n})$ in (\ref{defT}).
\end{proof}

The sequential procedure for assigning weights to pairings is illustrated in Figure~\ref{chordsDyck}. We are now ready to calculate the joint mixed moments of interest.

\begin{figure}\centering
\includegraphics[scale=0.5]{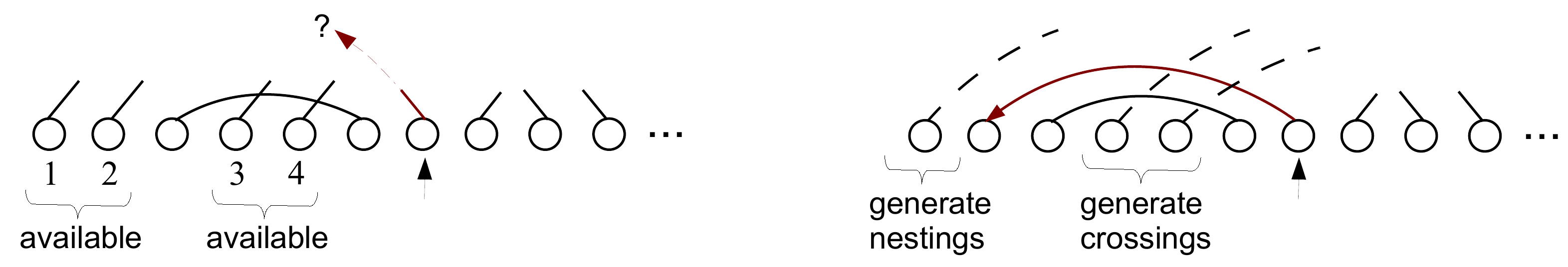}
\caption{The sequential pairing of openings and closures corresponding to the Dyck path of Figure~\ref{pathDyck}, with an arrow denoting the currently considered closure. In the left figure, there are four available (i.e. as of yet unpaired) openings. In the right figure, the current closure is paired to the indicated opening, thus incurring two crossings and one nesting. Note that the crossings and nestings incurred by the current closure never include any crossings or nestings already counted in the previous closure.}
\label{chordsDyck}
\end{figure}

\begin{lem} For all $n\in\mathbb N$ and $\epsilon(1),\ldots,\epsilon(2n)\in \{1,\ast\}^{2n}$,
\begin{flushleft} $\displaystyle\varphi_{q,t}(a(h_1)^{\epsilon(1)}\ldots a(h_{2n-1})^{\epsilon(2n-1)})=0$\\\end{flushleft}
\begin{align*}\displaystyle&\varphi_{q,t}(a(h_1)^{\epsilon(1)}\ldots a(h_{2n})^{\epsilon(2n)})=\left\{\begin{array}{ll}0,&(\epsilon(2n),\ldots,\epsilon(1))\not\in \psi_{2n}^{-1}(D_n)\\
T_{q,t}(h_{2n},\ldots,h_{1};\epsilon(2n),\ldots,\epsilon(1)),&(\epsilon(2n),\ldots,\epsilon(1))\in \psi_{2n}^{-1}(D_n)\end{array}\right.\\
&=\sum_{\mathscr V\in \mathscr P_{2}(2n)}\varphi_{q,t}(a(h_{w_1})^{\epsilon(w_1)}a(h_{z_1})^{\epsilon(z_1)})\ldots\varphi_{q,t}(a(h_{w_n})^{\epsilon(w_n)}a(h_{z_n})^{\epsilon(z_n)})q^{\cross(\mathscr V)}t^{\nest(\mathscr V)},\end{align*}
where each $\mathscr V$ is (uniquely) written as a collection of pairs $\{(w_1,z_1),\ldots,(w_n,z_n)\}$ with $w_1<\ldots<w_n$ and $w_i<z_i$.
\label{momentsC}
\end{lem}
\begin{proof}
Given the mixed moment $\varphi_{q,t}(a(h_1)^{\epsilon(1)}\ldots a(h_{k})^{\epsilon(k)})$, consider
the reverse string $\epsilon(k),\epsilon(k-1),\ldots,\epsilon(1)$ and the corresponding NE/SE path $\psi_k(\epsilon(k),\epsilon(k-1),\ldots,\epsilon(1))$. 
Since $a(h)\Omega=0$ for all $h\in\H$, by recursively expanding the mixed moment $\varphi_{q,t}(a(h_1)^{\epsilon(1)}\ldots a(h_{k})^{\epsilon(k)})$ via (\ref{cdef1}) and (\ref{cdef2}),
it immediately follows that the moment is zero if for any $1\leq i\leq k$, the number of SE steps in $\psi_i(\epsilon(k),\epsilon(k-1),\ldots,\epsilon(i))$ exceeds the corresponding number of NE steps. Moreover, since for all $h_i,g_i\in\H$, $\langle h_1\otimes\ldots\otimes h_n,g_1\otimes\ldots\otimes g_m\rangle_{q,t}=0$ whenever $n\neq m$, it follows that the moment vanishes unless the total number of NE steps in $\psi_k(\epsilon(k),\ldots,\epsilon(1))$ equals the corresponding number of SE steps. This shows that the moment $\varphi_{q,t}(a(h_1)^{\epsilon(1)}\ldots a(h_{k})^{\epsilon(k)})$ vanishes if $\psi_k(\epsilon(k),\ldots,\epsilon(1))$ is not a Dyck path. In particular, the mixed moment vanishes if $k$ is odd.

Now let $k$ be even with $\psi(\epsilon(k),\ldots,\epsilon(1))\in D_n$.
For some $m\geq 1$, let $\epsilon(m+1)$ correspond to the first $\ast$ from the left in the reverse string $(\epsilon(k),\ldots,\epsilon(1))$ (that is, corresponding to the right-most creation operator in $a(h_1)^{\epsilon(1)}\ldots a(h_{k})^{\epsilon(k)}$). Then, by (\ref{cdef1}), $a(h_{m+1})$ acts on $h_{m}\otimes\ldots\otimes h_1$ to produce a weighted sum of $(m-1)$-dimensional products, that is,
$$a(h_{m+1})h_{m}\otimes\ldots\otimes h_1=\sum_{i=1}^mq^{m-i}t^{i-1}\langle h_{m+1},h_i\rangle_\H h_{m}\otimes\ldots\breve h_i\otimes\ldots\otimes h_1.$$
At the same time, diagramatically, $\epsilon(m+1)$ corresponds to the the first closure from the left in $\psi(\epsilon(k),\ldots,\epsilon(1))$. Supposing that this closure pairs to the $i^\text{th}$ opening (from the left), for $1\leq i\leq m$, the weight of the resulting pair in the sense of Lemma~\ref{lemrecursive} is then given by $q^{m-i}t^{i-1}\langle h_{m+1},h_i\rangle_\H$. Furthermore, the act of removing $h_i$ from the product $h_{m}\otimes\ldots\otimes h_1$ diagramatically corresponds to removing the previously completed pairs in the procedure of Lemma~\ref{lemrecursive} and, in both cases, the same iteration is subsequently repeated on the thus reduced object. Now, by definition, summing the weights $\text{wt}(\mathscr V;h_k,\ldots,h_1)$ over all pairings $\mathscr V\sim_p(\epsilon(k),\ldots,\epsilon(1))$ is equivalent to summing the products of the weights of the individual pairs over all the possible ways of matching all closures to openings (and thus, in the above notation, over all choices of $i$ and analogous choices made on the subsequent iterations). Thus, 
$$\sum_{\substack{\mathscr V\in \mathscr P_{2}(2n)\\\mathscr V\sim_p (\epsilon(k),\ldots,\epsilon(1))}}\text{wt}(\mathscr V;h_k,\ldots,h_{1})$$
is exactly the sum of weights obtained by unfolding the expression $\varphi_{q,t}(a(h_1)^{\epsilon(1)}\ldots a(h_{k})^{\epsilon(k)})$ via the recursive definitions (\ref{cdef1}) and (\ref{cdef2}).
In other words, we have shown that 
$$\varphi_{q,t}(a(h_1)^{\epsilon(1)}\ldots a(h_{2n})^{\epsilon(2n)})=T_{q,t}(h_{2n},\ldots,h_{1};\epsilon(2n),\ldots,\epsilon(1))$$
whenever $(\epsilon(2n),\ldots,\epsilon(1))\in \psi_{2n}^{-1}(D_n)$.

Finally, that 
$\varphi_{q,t}(a(h_1)^{\epsilon(1)}\ldots a(h_{2n})^{\epsilon(2n)})$ also equals
$$\sum_{\mathscr V\in \mathscr P_{2}(2n)}\varphi_{q,t}(a(h_{w_1})^{\epsilon(w_1)}a(h_{z_1})^{\epsilon(z_1)})\ldots\varphi_{q,t}(a(h_{w_n})^{\epsilon(w_n)}a(h_{z_n})^{\epsilon(z_n)})\,q^{\cross(\mathscr V)}t^{\nest(\mathscr V)}$$
follows immediately from the fact that $\varphi_{q,t}(a(h_{w_j})^{\epsilon(w_j)}a(h_{z_j})^{\epsilon(z_j)})=0$ unless $\epsilon(w_j)=1$ and $\epsilon(z_j)=\ast$; in other words, unless $\mathscr V \sim_p (\epsilon(2n),\ldots,\epsilon(1))$.

\end{proof}

In particular, the mixed moments of the single element $a(h)$ are of a particularly insightful form.
\begin{lem} Given a Dyck path $\psi_n(\epsilon(2n),\ldots,\epsilon(1))$ and $h\in\H$, let $\widetilde{\text{wt}}(\epsilon(2n),\ldots,\epsilon(1);h)$ denote the weight of the path taken as the product of the weights of the individual steps, with each NE step assigned unit weight and each SE step falling from height $m$ to height $m-1$ assigned weight
\begin{equation}[m]_{q,t}:=\sum_{i=1}^{m} q^{m-i}t^{i-1}=\frac{t^m-q^m}{t-q}.\end{equation}
(See the illustration of Figure~\ref{labeledDyck}.) Then, $\displaystyle \varphi_{q,t}(a(h)^{\epsilon(1)}\ldots a(h)^{\epsilon(2n)})=\widetilde{\text{wt}}(\epsilon(2n),\ldots,\epsilon(1);h).$\label{lemproduct}
\end{lem}
\begin{proof}
Returning to the sequential procedure in the proof of Lemma~\ref{lemrecursive}, note that when $h_1=\ldots=h_{2n}=h$, the choice of an opening for any given closure only affects the weight of the present pair and does not affect that of the subsequently considered pairs. In particular, for any given closure and $m$ available (previously unpaired) openings to the left of it, the generating function of the weight of current pair is therefore given by $\sum_{i=1}^m q^{m-i}t^{i-1}\|h\|^2_\H$. Moreover, by the same token, the sum of weights of all the pairings $\mathscr V$ in the equivalence class of a fixed $(\epsilon(2n),\ldots,\epsilon(1))\in \psi_n^{-1}(D_n)$ is given by the product of the corresponding generating functions. Finally, note that for any given closure, the ``$m$'' is determined by the underlying $(\epsilon(2n),\ldots,\epsilon(1))$; specifically, the reader may readily verify that $m$ is exactly the height of the Dyck path preceding the corresponding given SE step encoding the given closure. 
\end{proof}

Figure~\ref{labeledDyck} provides an example of assigning weights to the steps of the Dyck path given by $\psi_{14}(\ast,\ast,\ast,\ast,\ast,1,1,\ast,1,1,1,1,\ast,1)$ according to the rules defined in the above proof. In particular, in light of Lemma~\ref{lemproduct}, the product of the weights of the individual steps then yields the moment $$\varphi_{q,t}(a(h) a(h)^\ast  (a(h))^4 a(h)^\ast (a(h))^2 (a(h)^\ast)^5).$$
\begin{figure}\centering
\includegraphics[scale=0.7]{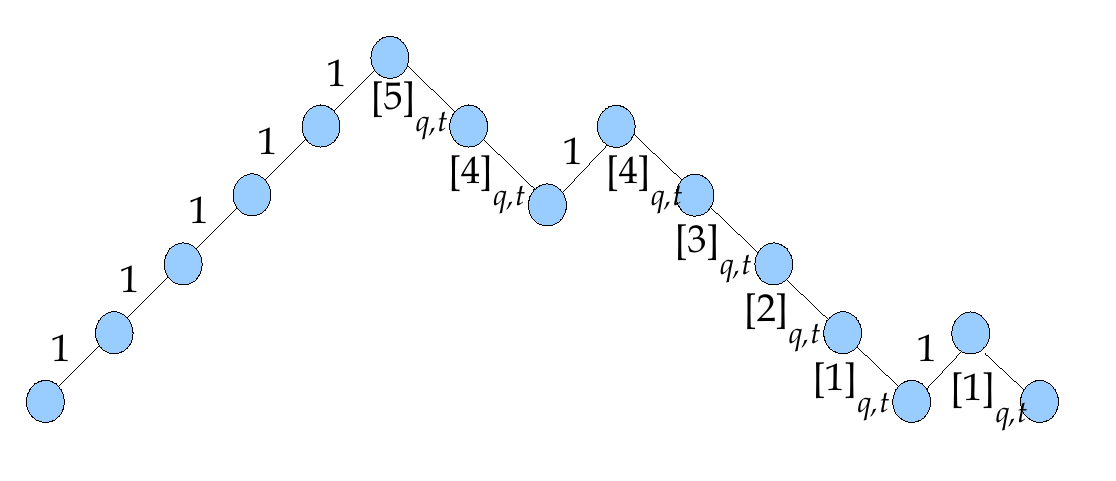}
\caption{Weighted path corresponding to Figure~\ref{pathDyck} for $\|h\|_\H=1$.}\label{labeledDyck}
\end{figure}

At this point it is also interesting to note the combinatorial statistics featuring prominently in the previous section, given by inversions and coinversions of permutations, are intimately related to those of crossings and nestings in pair partitions considered presently. In particular, by Lemma~\ref{lemP},
$$\langle g_n\otimes\ldots\otimes g_1,h_1\otimes\ldots\otimes h_n\rangle_{q,t}=\sum_{\pi\in S_n}q^{\inv(\pi)}t^{\div(\pi)}\langle g_n,h_{\pi(1)}\rangle_\H\ldots\langle g_1,h_{\pi(n)}\rangle_\H,$$ 
while, at the same time,
\begin{eqnarray*}\langle g_n\otimes\ldots\otimes g_1,h_1\otimes\ldots\otimes h_n\rangle_{q,t}&=&\langle\Omega,a(g_1)\ldots a(g_n)a^\ast(h_1)\ldots a(h_n)\Omega\rangle_{q,t}\\
&=&\sum_{\substack{\mathscr V\in \mathscr P_{2}(2n)\\\mathscr V\sim_p (1,1,\ldots,1,\ast,\ast,\ldots,\ast)}}q^{\cross(\mathscr V)}t^{\nest(\mathscr V)}\prod_{i=1}^n\langle g_{w_i},h_{z_i}\rangle_\H,
\end{eqnarray*}
where the second equality follows by Lemma~\ref{momentsC}. In other words, the value of $\langle g_n\otimes\ldots\otimes g_1,h_1\otimes\ldots\otimes h_n\rangle_{q,t}$ is computed by summing the weights over pairings $\mathscr V$ in the equivalence class of $(1,1,\ldots,1,\ast,\ast,\ldots,\ast)$. Writing each such $\mathscr V$ in the two-line notation of the previous section, with the $\epsilon(i)=1$ corresponding to the top row and $\epsilon(j)=\ast$ to the bottom row, the reader may readily verify that the crossings turn into inversions and nestings into coinversions, and vice-versa. 

\subsection{Joint and marginal statistics of $(q,t)$-Gaussian elements}

Analogously to both classical and free probability and, more generally, to the non-commutative probability constructed over the $q$-Fock space of~\cite{Bozejko1991}, the so-called ``Gaussian elements'' will turn out to occupy a fundamental role in the non-commutative probability built over the $(q,t)$-Fock space. More precisely, for $h\in\H$, the $(q,t)$-Gaussian element $s_{q,t}(h)$ is defined as the (self-adjoint) element of $\mathscr{G}_{q,t}$ given by $a(h)+a(h)^\ast$. In \cite{Blitvic2}, the $(q,t)$-Gaussian family will be tied to an extended non-commutative Central Limit Theorem and, in particular, will be identified with the limit of certain normalized sums of random matrix models. Meanwhile, note that the joint moments of the $(q,t)$-Gaussian elements come in a form that is the $(q,t)$-analogue of the familiar Wick formula (e.g. \cite{NicaSpeicher}).

\begin{defn} For $h\in\H$, let the \emph{$(q,t)$-Gaussian} element $s_{q,t}(h) \in \mathscr{G}_{q,t}$ be given by $s_{q,t}(h):=a(h)+a(h)^\ast$.\label{defsemicircular}
\end{defn}

\begin{lem}[The $q,t$-Wick Formula] For all $n\in\mathbb N$ and $h_1,\ldots,h_{2n}\in\H$,
\begin{eqnarray*}\varphi_{q,t}(s_{q,t}(h_1)\ldots s_{q,t}(h_{2n-1}))&=&0\\
\varphi_{q,t}(s_{q,t}(h_1)\ldots s_{q,t}(h_{2n}))&=& \sum_{\mathscr V=\{(w_1,z_1),\ldots(w_n,z_n)\}\in \mathscr P_{2}(2n)} q^{\cross(\mathscr V)}\,t^{\nest(\mathscr V)}\,\prod_{i=1}^n\langle h_{w_i},h_{z_i}\rangle_\H.
\end{eqnarray*}\label{qtWick}
\end{lem}
\begin{proof} 
It suffices to note that
$$\varphi_{q,t}(s_{q,t}(h_1)\ldots s_{q,t}(h_{n}))=\sum_{\epsilon(1),\ldots,\epsilon(n)\in\{1,\ast\}}\varphi_{q,t}(a(h_1)^{\epsilon(1)}\ldots a(h_{2n})^{\epsilon(2n)}),$$
from which the result follows directly by Lemma~\ref{momentsC} and (\ref{defT}).
\end{proof}

Focusing on a single element $s_{q,t}(h)$, the corresponding moments can be expressed in a particularly elegant form. In particular, in light of Lemma~\ref{lemproduct}, the moment $\varphi_{q,t}(s_{q,t}(h)\ldots s_{q,t}(h))$ is expressible via generating functions of weighted Dyck paths, which one can then interpret as a continued fraction via a well-known correspondence (e.g. \cite{Flajolet1980}). Note that the continued-fraction formulation in the following lemma already features in \cite{Kasraoui2006}, in a more general combinatorial context.

\begin{lem} The moments of the $(q,t)$-Gaussian element $s_{q,t}(h)$ are given by
\begin{eqnarray*}
\varphi_{q,t}(s_{q,t}(h)^{2n-1})&=&0\\
\varphi_{q,t}(s_{q,t}(h)^{2n})&=&\|h\|_\H^{2n} \sum_{\mathscr V\in \mathscr P_2(2n)}  \!\!\!q^{\cross(\mathscr V)}\,t^{\nest(\mathscr V)}
=\,\|h\|_\H^{2n}\,\,[z^n]\,\cfrac{1}{1-\cfrac{[1]_{q,t}z}{1-\cfrac{[2]_{q,t}z}{1-\cfrac{[3]_{q,t}z}{\ldots}}}},
\end{eqnarray*}
where $[z^n](\cdot)$ denotes the coefficient of the $z^n$ term in the power series expansion of $(\cdot)$.
\label{lemsemicircular}
\end{lem}
\begin{proof}
\item Expanding the moment and applying Lemma~\ref{lemproduct},
\begin{eqnarray*}\varphi_{q,t}(s_{q,t}(h)^{2n})&=&\sum_{\epsilon(1),\ldots,\epsilon(n)\in\{1,\ast\}}\text{wt}(\epsilon(2n),\ldots,\epsilon(1);h,\ldots,h)\\&=&\sum_{(\epsilon(k),\ldots,\epsilon(1))\in\psi^{-1}(D_n)} \widetilde{\text{wt}}(\epsilon(k),\ldots,\epsilon(1);h).
\end{eqnarray*}
The continued fraction expansion of $\sum_{n\geq 0}\varphi_{q,t}(s_{q,t}(h)^{2n})z^n$ is then obtained by the classic encoding of weighted Dyck paths (see \cite{Flajolet1980}, also in a more relevant context \cite{Biane1997} and \cite{Kasraoui2006}).
\end{proof}

By the previous lemma, the first four even moments of $s_{q,t}(h)$ are thus given by
\begin{eqnarray*}\varphi_{q,t}(s_{q,t}(h)^{0})&=&1\\
\varphi_{q,t}(s_{q,t}(h)^{2})&=&\|h\|_\H^2\\
\varphi_{q,t}(s_{q,t}(h)^{4})&=& \|h\|_\H^4\,(1 + q + t)\\
\varphi_{q,t}(s_{q,t}(h)^{6})&=&  \|h\|_\H^6\,(1 + 2q + 2t + 2qt + q^2+ t^2+ 2q^2t + 2qt^2+ q^3+ t^3),
\end{eqnarray*}
may be verified by enumerating all chord-crossing diagrams with, respectively, between 0 and 3 chords and counting the corresponding crossings and nestings.

\begin{remark} Choosing $\|e\|_\H=1$, $\varphi_{q,t}(s_{q,t}(e)^{2n})$ is equal to the generating function of crossings and nestings in $\mathscr P_2(2n)$, that is,
$$\varphi_{q,t}(s_{q,t}(e)^{2n})=\sum_{\mathscr V\in \mathscr P_2(2n)}  \!\!\!q^{\cross(\mathscr V)}\,t^{\nest(\mathscr V)},$$
whose continued-fraction encoding, as previously noted, already explicitly features in \cite{Kasraoui2006}.

Since $[n]_{q,t}=[n]_{t,q}$, the continued-fraction expansion shows that the joint generating function is in fact symmetric in the two variables, that is,
$$\sum_{\mathscr V\in \mathscr P_2(2n)}  \!\!\!q^{\cross(\mathscr V)}\,t^{\nest(\mathscr V)}=\sum_{\mathscr V\in \mathscr P_2(2n)}  \!\!\!t^{\cross(\mathscr V)}\,q^{\nest(\mathscr V)}.$$
Given the fundamental constraint $|q|<t$, necessary for the positivity of the sesquilinear form $\langle,\rangle_{q,t}$, this symmetry is surprising. Indeed, let 
$$T_{q,t}(n):=\left\{\begin{array}{ll}\sum_{\mathscr V\in \mathscr P_2(n)}  q^{\cross(\mathscr V)}\,t^{\nest(\mathscr V)}&n\text{ even}\\0&n\text{ odd}\end{array}\right.$$ and note that associating $T_{q,t}(n)$ with the moments of a self-adjoint element (viz. $s_{q,t}(e)$) in a $C^\ast$ probability space ensures the positivity of $(T_{q,t}(n))_{n\in\mathbb N}$ (as a number sequence) in the range $|q|<t$ \cite{Akhiezer}. The above symmetry then yields the positivity of $T_{t,q}(n)$ for $|q|<t$, which, in turn, guarantees the existence of a real measure whose moments are given by $T_{t,q}(n)$ \cite{Akhiezer}. This observation naturally beckons the construction of spaces that would give rise to such ``reflected'' Gaussian algebras. 
\end{remark}

\begin{remark}
Letting $t=1$ in the moment formulation of Lemma~\ref{lemsemicircular} yields an analogous continued fraction with $[n]_{q}$ playing the role of the coefficients $[n]_{q,t}$. What is more, the corresponding generating function admits an interesting explicit (even if not closed-form) expression, namely
$$\sum_{\mathscr V\in \mathscr P_2(2n)}  \!\!\!q^{\cross(\mathscr V)}=\frac{1}{(1-q)^n}\sum_{k=-n}^n(-1)^kq^{k(k-1)/2}{{2n}\choose{n+k}},$$
known as the \emph{Touchard-Riordan formula} \cite{Touchard1952,Touchard1950,Touchard1950-2,Riordan1975}.
\end{remark}

The general form of continued fraction appearing in Lemma~\ref{lemsemicircular} is known as the Stieltjes-fraction ($S$-fraction). Specifically, an $S$-fraction is written as $\frac{1|}{|1}+\frac{\lambda_1 z|}{|1}+\frac{\lambda_2 z|}{|1}+\frac{\lambda_3 z|}{|1}+\ldots$, where, in the present setting, $\lambda_n=[n]_{q,t}$. In the combinatorial approach to orthogonal polynomials pioneered by Flajolet and Viennot \cite{Flajolet1980,Viennot1985}, an $S$-fraction $\frac{1|}{|1}+\frac{\lambda_1 z|}{|1}+\frac{\lambda_2 z|}{|1}+\frac{\lambda_3 z|}{|1}+\ldots$ and the orthogonal polynomial sequence $\{y_n(z)\}_{n\geq 0}$ given by $y_0(z)=1$, $y_1(z)=1$, and $y_{n+1}(z)=zy_n(z)-\lambda_n y_{n-1}(z)$ can both be considered as encodings of a certain object. Namely, they both encode a collection of Dyck paths with NE steps carrying unit weight and each SE step of height $n\mapsto n-1$ carrying weight $\lambda_n$. Returning to the setting at hand, consider therefore the following (further) deformation of the classical and quantum Hermite orthogonal polynomials.

\begin{defn}
The $(q,t)$-Hermite orthogonal polynomial sequence $\{H_n(z;q,t)\}_{n\geq 0}$ is determined by the following three-term recurrence:
$$zH_n(z;q,t)=H_{n+1}(z;q,t)+[n]_{q,t}H_{n-1}(z;q,t),$$
with $H_0(z;q,t)=1,\,\,H_1(z;q,t)=z.$
\label{qtHermite}
\end{defn}

While the following Lemma~\ref{lem_qt_Hermite} can also be deduced from the classical theory (e.g. \cite{Akhiezer}), the combinatorial approach of \cite{Viennot1985} reveals the rich structure associated with orthogonal polynomial sequences and provides a lucid link between continued fractions, moment sequences, and orthogonal polynomials.   
In particular, interpreting the underlying Hankel determinants as certain sums of weighted paths, the $S$-type fraction  $\frac{1|}{|1}+\frac{\lambda_1 z|}{|1}+\frac{\lambda_2 z|}{|1}+\frac{\lambda_3 z|}{|1}+\ldots$ is readily shown to be the continued-fraction encoding of the generating function $\sum_{n\geq 0}m_nz^n$, where $(m_n)_{n\geq 0}$ is the moment sequence associated with an orthogonalizing linear functional for the polynomial sequence $y_0(z)=1$, $y_1(z)=1$, $y_{n+1}(z)=zy_n(z)-\lambda_n y_{n-1}(z)$. In the present context, in the light of Lemma~\ref{lemsemicircular}, this correspondence yields that the $(q,t)$-Hermite polynomial sequence is orthogonal with respect to the linear functional given as the integral against the measure of $s_{q,t}(e)$. In turn, since $s_{q,t}(e)$ is bounded for $|q|< t\leq 1$ (as the real part of a bounded operator, cf. Lemma~\ref{lemmanorm}), the corresponding measure is compactly supported and is therefore uniquely determined by its moment sequence. 

To summarize:
\begin{lem}
The distribution of the $(q,t)$-Gaussian element $s_{q,t}(e)$, where $\|e\|_\H=1$, is the unique real measure that orthogonalizes the $(q,t)$-Hermite orthogonal polynomial sequence.\label{lem_qt_Hermite}
\end{lem}

\subsection{The $(\Gamma_{q,t},\varphi_{q,t})$ $\ast$-probability space}

While the pair $(\mathscr{G}_{q,t},\varphi_{q,t})$ provides a natural setting in which one can discuss the mixed moments of the creation operators and field operators,  it is not an especially well-behaved non-commutative probability space. In particular, the functional $\varphi_{q,t}$ is neither faithful nor tracial, as $\varphi_{q,t}(a(h)^\ast a(h))=0\neq \varphi_{q,t}(a(h) a(h)^\ast)=\|f\|_\H$. As in the full Fock space and, more generally, the $q$-Fock space, one may wish instead to focus on a subalgebra of non-commutative random variables generated by the $(q,t)$-Gaussian elements of Definition~\ref{defsemicircular}. In particular, consider the unital $\ast$-algebra generated by $\{s_{q,t}(h)\mid h\in\H\}$ and let $\Gamma_{q,t}$ denote its weak closure in $\mathscr B(\F_{q,t})$. In other words, $\Gamma_{q,t}$ is the $(q,t)$-Gaussian von Neumann algebra. While the vacuum expectation state on the corresponding algebra on the $q$-Fock space, corresponding to the $t=1$ case in the present setting, is tracial \cite{Bozejko1994}, the same does not hold in the more general case.

\begin{prop} For $\text{dim}(\H)\geq 2$, $\Gamma_{q,t}$ is tracial if and only if $t=1$.\label{lemtrace}
\end{prop}
\begin{proof} The forward direction is established in \cite{Bozejko1994} in a more general setting. For the converse, it suffices to consider some four vectors $h_1,h_2,h_3,h_4\in\F_{q,t}$. Using Lemma~\ref{qtWick}, 
$$\varphi_{q,t}(s_{q,t}(h_1)s_{q,t}(h_2)s_{q,t}(h_3)s_{q,t}(h_4))_{q,t}=\langle h_1,h_2\rangle_\H\langle h_3,h_4\rangle_\H+t\langle h_1,h_4\rangle_\H\langle h_2,h_3\rangle_\H+q\langle h_1,h_3\rangle_\H\langle h_2,h_4\rangle_\H$$
and
$$\varphi_{q,t}(s_{q,t}(h_4)s_{q,t}(h_1)s_{q,t}(h_2)s_{q,t}(h_3))_{q,t}=\langle h_4,h_1\rangle_\H\langle h_2,h_3\rangle_\H+t\langle h_4,h_3\rangle_\H\langle h_1,h_2\rangle_\H+q\langle h_4,h_2\rangle_\H\langle h_1,h_3\rangle_\H.$$
Since $\H$ is a real Hilbert space, the terms in $q$ are equal in both of the above expressions, but the terms in $1$ and $t$ are interchanged. Denoting by $\{e_i\}$ the basis of $\H$ and taking $h_1=h_2=e_1$ and $h_3=h_4=e_2$, it follows that the two expressions are equal if and only if $t=1$.
\end{proof}

\begin{remark}
The fact that the term in $q$ in the previous example remained the same for both $\langle s_{q,t}(h_1)s_{q,t}(h_2)s_{q,t}(h_3)s_{q,t}(h_4)\Omega\rangle_{q,t}$ and $\langle\Omega,s_{q,t}(h_4)s_{q,t}(h_1)s_{q,t}(h_2)s_{q,t}(h_3)\Omega\rangle_{q,t}$ is not coincidental. Indeed, from a combinatorial viewpoint, commuting two products of Gaussian elements is equivalent to rotating (by a fixed number of positions) the chord diagrams corresponding to each of the non-vanishing products of the underlying creation and annihilation operators. The observed equality then follows from the fact that the crossings in a chord diagram are preserved under diagram rotations. The same is however not true of nestings, which is the combinatorial reason for the overall loss of traciality of $\varphi$ for $t<1$. 
\end{remark}

\section{Case $0=q<t$: ``$t$-deformed Free Probability''}
\label{t-free}
The remainder of the paper considers the case $0=q<t\leq 1$, corresponding to a new single-parameter deformation of the full Boltzmann Fock space of free probability~\cite{Voiculescu1986,Voiculescu1992}. 
Once again, this deformation will turn out to be particularly natural. In particular, the statistics of the $t$-deformed semicircular element will be described in an elegant form afforded by the deformed Catalan numbers of Carlitz and Riordan \cite{Carlitz1964}, the Rogers-Ramanujan continued fraction, and the $t$-Airy function of Ismail \cite{Ismail2005}. Moreover, this setting will give rise to a natural counterpart of the result of Wigner~\cite{Wigner1955}, as the $t$-deformed semicircular element will be seen to encode the first-order statistics of correlated Wigner processes. Prior to delving into these properties, we take a moment to show that the von Neumann algebra of bounded linear operators on $\F_{0,t}(\H)$ is generated by the creation operators $\{a_{0,t}(h)\}_{h\in\H}$.

\begin{lem} For $q=0<t\leq 1$, the von Neumann algebra $\mathscr W_{0,t}$ generated by $\{a_{0,t}(h)\}_{h\in\H}$ is $\mathscr B(\F_{0,t}(\H))$.\label{generateB}
\end{lem}
\begin{proof} The proof of this fact follows along analogous lines to the full Boltzmann Fock space setting, namely the $q=0$, $t=1$ case. For concreteness, the following sketch adapts the proof of Theorem 2.2 of \cite{Kemp2005}. In particular, let $\{e_1,e_2,\ldots\}$ denote an orthonormal basis for $\H$ and consider the operator 
$P=t^{-N+1}\sum_{j=1}^\infty a(e_j)^\ast a(e_j)\,\in\mathscr W_{0,t}.$
Since 
$$a(e_j)\Omega=0\quad \text{and} \quad a(e_j)e_{i_1}\otimes\ldots\otimes e_{i_n}=\left\{\begin{array}{ll}0,&i_1\neq j\\t^{n-1}e_{i_2}\otimes\ldots\otimes e_{i_n},&i_1=j\end{array}\right.,$$ where $a(h):=a_{0,t}(h)$, it follows that $P\Omega=0$ and $P\,e_{i_1}\otimes\ldots\otimes e_{i_n}=e_{i_1}\otimes\ldots\otimes e_{i_n}$. Therefore, $P$ is the projection onto the closure of $\bigoplus_{n\geq 1}\mathscr H_{\mathbb C}^{\otimes n}$. Since $\langle \xi,\eta\rangle_{q,t}=0$ whenever $\xi\in\mathscr H^{\otimes n}$ and $\eta \in\mathscr H^{\otimes m}$ for $n\neq m$ ($n,m\geq 0$), $P$ is the projection onto the orthogonal complement of the vacuum and $P_\Omega:=I-P\in\mathscr W_{0,t}$ is the projection onto the vacuum. Thus, $\mathscr W_{0,t}$ contains the operator
$a(e_{i_1})^\ast\ldots a(e_{i_n})^\ast P_{\Omega}\, a(e_{j_1})\ldots a(e_{j_m}),$
which is a rank-one operator with image spanned by $e_{i_1}\otimes\ldots\otimes e_{i_n}$ and kernel orthogonal to $e_{j_1}\otimes\ldots\otimes e_{j_m}$. It follows that $\mathscr W_{0,t}$ contains all finite-rank operators. Taking closures, $\mathscr W_{0,t}\supseteq\mathscr B(\F_{0,t}(\H))$ and the result follows.
\end{proof}

\subsection{The $t$-semicircular element} In light of its present interpretation in the context of deformed free probability, the \emph{$t$-semicircular element} is a renaming of the $(0,t)$-Gaussian element $s_{0,t}(h) \in \mathscr{G}_{0,t}$. Specializing accordingly the spectral properties derived in the previous section, the $t$-semicircular element turns out to encode certain familiar objects in combinatorics and number theory.

\begin{lem} The moments of the $t$-semicircular element $s_{0,t}(h)$ are given by
\begin{eqnarray*}
\varphi_{0,t}(s_{0,t}(h)^{2n-1})&=&0\\
\varphi_{0,t}(s_{0,t}(h)^{2n})&=&\|h\|_\H^{2n} \sum_{\mathscr V\in NC_2(2n)} t^{\nest(\mathscr V)}=\|h\|_\H^{2n}\,C_n^{(t)}
\end{eqnarray*}
where $NC_2(2n)$ denotes the lattice of non-crossing pair-partitions and $C_n^{(t)}$ are referred to as the Carlitz-Riordan $t$-Catalan (or, rather, $q$-Catalan\footnote{In keeping with the original objects in \cite{Carlitz1964}, the $q$-Catalan numbers of Carlitz-Riordan are more frequently given by $q^{-{n\choose 2}}C_n^{(1/q)}$, i.e. with reversed coefficients compared to the present definition.}) numbers\cite{Furlinger1985,Carlitz1964}, given by the recurrence
\begin{equation}
C_n^{(t)}=\sum_{k=1}^n t^{k-1} C_{k-1}^{(t)} C_{n-k}^{(t)},\label{qCat}
\end{equation}
with $C_0^{(t)}=1$.
\label{lemtsemicircular}
\end{lem}
\begin{proof} The first two equalities are obtained by Lemma~\ref{lemsemicircular}, substituting $q=0$. Next, set $\alpha_n:=\sum_{\mathscr V\in NC_2(2n)} t^{\nest(\mathscr V)}$ and note that $\alpha_0=1$. For $\pi\in NC_2(2n)$, consider the pair $(1,\beta)\in\pi$ and note that, since $\pi$ is non-crossing, (1) $\beta$ is even, and (2) the remaining pairs either belong to the interval $I=\{2,\ldots,\beta-1\}$ or to $I'=\{\beta+1,\ldots,n\}$. Thus, given a nesting in $\pi$, either both participating pairs belong to $I$ or both belong to $I'$ or one of those pairs is in fact equal to $\{1,\beta\}$. In the latter case, the nesting must be formed by drawing a second pair from $I$. Summing over all $\pi\in NC_2(2n)$ and conditioning on the choice of $2\beta\in \{1,2,\ldots,n\}$ recovers the recurrence in (\ref{qCat}). The corresponding argument is illustrated in Figure~\ref{recurrence}.
\end{proof}

\begin{figure}\centering
\includegraphics[scale=0.55]{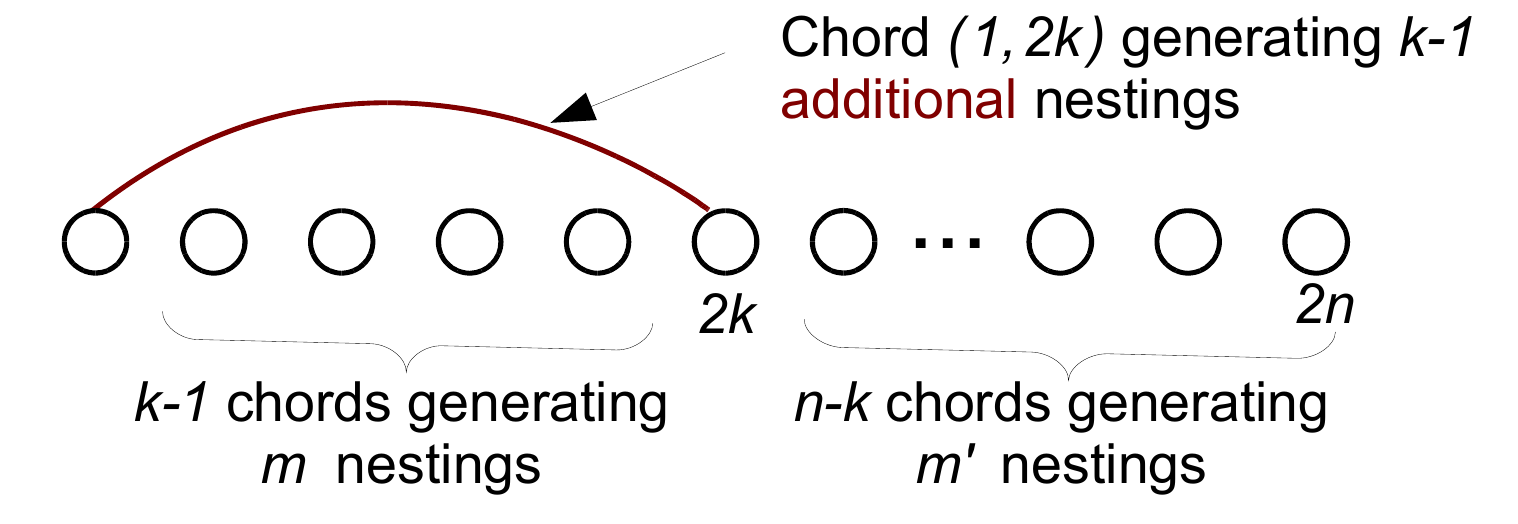}
\caption{Recurrence (\ref{qCat}): given $\pi\in NC_2(2n)$ with $(1,2k)\in\pi$ such that the chords contained in the interval $I=\{2,\ldots,2k-1\}$ generate $m$ nestings and those contained in $I'=\{2k+1,\ldots,n\}$ generate $m'$ nestings, the total number of nestings in $\pi$ is given by $m+m'+k-1$.}\label{recurrence}
\end{figure}

The first few even moments of $s_{0,t}(h)$ are thus given by $\varphi_{q,t}(s_{q,t}(h)^{2})=\|h\|_\H^2$, $\varphi_{q,t}(s_{q,t}(h)^{4})=\|h\|_\H^2(1+t)$, $\varphi_{q,t}(s_{q,t}(h)^{6})=\|h\|_\H^2(1+2t+t^2+t^3)$. The deformed Catalan numbers of the above lemma are known to be related to, among other statistics, inversions of Catalan words, Catalan permutations, and area below lattice paths \cite{Furlinger1985}. 

Rather than further considering individual moments, specializing to $q=0$ the continued fraction encoding of the $(q,t)$-Gaussian measure yields a more fascinating object still. Note that for notational convenience, and without loss of generality, the remainder of the section considers the normalized $t$-semicirculars canonically written as $s_{0,t}:=s_{0,t}(e)$, for some unit vector $e\in \H$.

\begin{lem}
The moments of the normalized $t$-semicircular element $s_{0,t}$ are encoded by the Rogers-Ramanujan continued fraction as
$$\sum_{n\geq 0}\varphi_{0,t}(s_{0,t}(h)^{n})z^n=\cfrac{1}{1-\cfrac{t^0\,z}{1-\cfrac{t^1\,z}{1-\cfrac{t^2\,z}{\ldots}}}}$$

The Cauchy transform of the $t$-semicircular measure $\mu_{0,t}$ associated with $s_{0,t}(e)$ is given by 
$$\int_{\R}\frac{1}{z-\eta}\,d\mu_{0,t}(\eta)=z^{-1}\,\frac{\sum_{n\geq 0} (-1)^n\,\frac{t^{n^2}}{(1-t)(1-t^2)\ldots(1-t^n)}\,z^{-n}}{\sum_{n\geq 0} (-1)^n\,\frac{t^{n(n-1)} }{(1-t)(1-t^2)\ldots(1-t^n)}\,z^{-n}}.$$
\end{lem}
\begin{proof} The continued fraction follows immediately from Lemma~\ref{lemsemicircular} by letting $q=0$. Since $s_{0,t}(e)$ is bounded, the measure $\mu_{0,t}$ is compactly supported. Therefore, the Cauchy transform of $\mu_{0,t}$ has the power series expansion 
$$\int_{\R}\frac{1}{z-\eta}\,d\mu_{0,t}(\eta)=\sum_{n\geq 0}\frac{\varphi_{0,t}(s_{0,t}^n)}{z^{n+1}}=\frac{1}{z}\,M(1/z),$$
where $M(z) =\sum_{n\geq 0}\varphi_{0,t}(s_{0,t}^n)z^n$ (e.g. \cite{NicaSpeicher}).
It is well known (e.g. \cite{Andrews}) that the above Rogers-Ramanujan continued-fraction can be written as
$$\displaystyle \frac{\sum_{n\geq 0} (-1)^n\,\frac{t^{n^2}}{(1-t)(1-t^2)\ldots(1-t^n)}\,z^n}{\sum_{n\geq 0} (-1)^n\,\frac{t^{n(n-1)} }{(1-t)(1-t^2)\ldots(1-t^n)}\,z^n}.$$
Recalling that the continued fraction encodes $M(z)$ and performing the required change of variable yields the desired Cauchy transform.
\end{proof}

In turn, the functions governing the numerator and denominator of the Cauchy-transform of $\mu_{0,t}$ turn out to have a fascinating interpretation as single-parameter deformations of the Airy function, a fact discovered by Ismail \cite{Ismail2005}. Specifically, let the \emph{$t$-Airy function} be given by
\begin{equation}
A_t(z)=\sum_{n\geq 0}\frac{t^{n^2}}{(1-t)\ldots(1-t^n)}(-z)^n.
\end{equation}
In other words:

\begin{cor}
$$\int_{\R}\frac{1}{z-\eta}\,d\mu_{0,t}(\eta)=\frac{1}{z}\,\frac{A_t(1/z)}{A_t(1/(zt))}.$$
\end{cor}

\begin{remark} While the function $A_q(z)=\sum_{n\geq 0}\frac{q^{n^2}}{(1-q)\ldots(1-q^n)}(-z)^n$ features at various points in Ramanujan's work, the nomenclature is more recent. It was proposed by Ismail \cite{Ismail2005} upon discovering the fact that, analogously to the Airy function in the classical case, the function $A_q$ is involved in the large degree Plancherel-Rotach-type asymptotics for the $q$-polynomials of the Askey-scheme.\end{remark}

Prior to seeking to further characterize the measure $\mu_{0,t}$, it is instructive to take a moment to consider its associated orthogonal polynomial sequence. Since the moments of $\mu_{0,t}$ are given by $t$-deformed Catalan numbers, the $q=0$ subfamily of the $(q,t)$-Hermite orthogonal polynomials of Definition~\ref{qtHermite} can thus also be viewed as a $t$-deformed version of the Chebyshev II orthogonal polynomials.

\begin{defn}
The $t$-Chebyshev II orthogonal polynomial sequence $\{U_n(z;t)\}_{n\geq 0}$ is determined by the following three-term recurrence:
$$zU_n(z;t)=U_{n+1}(z;t)+t^{n-1}U_{n-1}(z;t),$$
with
$$U_0(z;t)=1,\,\,U_1(z;t)=z.$$
\label{tChebyshev}
\end{defn}

\begin{remark} The orthogonal polynomial sequence encoded by the generalized Rogers-Ramanujan continued fraction was first considered by Al-Salam and Ismail in \cite{Al-Salam1983}, and the polynomial $U_n(z;t)$ is the special case of their polynomial $U_n(z;a,b)$ for $a=0,b=t$. Similarly to the previous section, it again follows from the classical theory \cite{Akhiezer} that the $t$-semicircular measure $\mu_{0,t}$ is the unique positive measure on $\mathbb R$ orthogonalizing the $t$-Chebyshev II orthogonal polynomial sequence.\end{remark}

The elegant form of the Cauchy transform of the previous lemma provides means of describing this measure via the zeros of the $t$-Airy function. This was indeed done in \cite{Al-Salam1983} and, adapted to the present setting, is formulated as follows.

\begin{lem}[Corollary 4.5 in \cite{Al-Salam1983}] Let $\{z_j\}_{j\in\mathbb N}$ denote the sequence of zeros of the rescaled $t$-Airy function $A_t(z/t)$. The measure $\mu_{0,t}$ is a discrete probability measure with atoms at
$$\pm \sqrt{t/z_j},\quad j\in\mathbb N$$
with corresponding mass
$$-\frac{A_t(z_j)}{2\,z_j\,A_t^\prime(z_j/t)},$$
where $A_t^\prime(z):=\frac{d}{dz}A_t(z)$. The unique accumulation point of $\mu_{0,t}$ is the origin.
\end{lem}

\subsection{First-order statistics of the Wigner process}

For the purpose of this section, consider a \emph{Wigner matrix} $W_N=[w_{i,j}]\in\mathcal M_N(L_{\infty-}(\R,\mathscr B,\mathbb P))$ to be a self-adjoint random matrix with elements $\{w_{i,j}\}_{1\leq i\leq j\leq N}$ forming a jointly independent collection of centered random variables with unit variance and uniformly bounded moments, i.e. for all $i,j\in [N]$, $\E(w_{i,j})=0$, $\E(|w_{i,j}|^2)=1$ and $\E(|w_{i,j}|^n)\leq c_n\in \R$ for all $n\in\mathbb N$. The asymptotics of the distribution of $W_N$ remain an object of extensive study, taking root in the work of Wigner~\cite{Wigner1955} and further evolving over the following decades. In particular, considering an expectation functional  $\varphi_N:\mathcal M_N(L_{\infty-}(\R,\mathscr B,\mathbb P))\to\R$ given by $\varphi_N=\frac{1}{N} \text{Tr}\otimes \E$, the following result is known as Wigner's Semicircle Law.

\begin{thm*}[e.g. Theorem 2.1.1 and Lemma 2.1.6 in \cite{RMT}] For $W_N$ a Wigner matrix, the empirical spectral measure of $\tilde W_N:=W_N/N$ converges in probability to the semicircle distribution, with
\begin{eqnarray*}\lim_{N\to\infty}\varphi_N(\tilde W_N^{2n-1})&=&0\\
\lim_{N\to\infty}\varphi_N(\tilde W_N^{2n})&=&\frac{1}{n+1}{2n\choose n}\end{eqnarray*}\label{Wigner}
\end{thm*}

In particular, $W_N$ converges in moments to the familiar semicircular element $s_{0,1}$. It is natural to expect that by introducing some correlation into Wigner's framework, some deformation of the semicircle law, with an analogous refinement of the Catalan numbers,  may be achieved. Indeed, this was obtained by Khorunzhy~\cite{Khorunzhy} and further developed by Mazza and Piau~\cite{Mazza2002} in relation to the following setup. Let a \emph{Wigner process} refer to the sequence of  $\{W_{N,\rho}(k)\}_{k\in\mathbb N}$ of Wigner matrices satisfying the following conditions:
\begin{itemize}
\item The moments of $W_{N,\rho}(k)$ are uniformly bounded in $k$, i.e. for all $n\in\mathbb N$ and all $i,j\in [N]$, $\E(|w_{i,j}(k)|^n)\leq c_n\in \R$, where $w_{i,j}(k)$ corresponds to the $(i,j)^\text{th}$ entry of the matrix $W_{N,\rho}(k)$ and $c_n$ does not depend on $k$.
\item The processes $(w_{i,j}(k))_{k\in\mathbb N}$ for $i\leq j$ form a triangular array of independent processes.
\item Each process $(w_{i,j}(k))_{k\in\mathbb N}$ is $\rho$-correlated, i.e. for some $|\rho|\leq 1$ and any $1\leq k\leq m$,
\begin{equation}
\E(w_{i,j}(k)w_{i,j}(m))=\rho^{m-k}.
\end{equation}
\end{itemize}
Note that for $\rho=0$, $\{W_{N,0}(\rho)\}$ is a sequence independent, identically distributed Wigner matrices, whereas for $\rho=1$, the situation reduces to having copies of the same matrix with $W_{N,1}(1)=W_{N,1}(k)$ for all $k\in \mathbb N$. Let $$B_{n,N}:=\varphi_N\left(\frac{W_{N,\rho}(1)}{N}\frac{W_{N,\rho}(2)}{N}\ldots \frac{W_{N,\rho}(n)}{N}\right).$$ 
The convergence in $N$ of the sequence $B_{n,N}$ was previously established in \cite{Khorunzhy,Mazza2002} and the corresponding limits computed explicitly. The surprise is that the limiting moments are, in fact, those of the $t$-semicircular element.

\begin{prop} For $n\in\mathbb N$ and $|\rho|\leq 1$,
\begin{eqnarray}
\lim_{N\to\infty}\varphi_N\left(\frac{W_{N,\rho}(1)}{N}\,\frac{W_{N,\rho}(2)}{N}\ldots \frac{W_{N,\rho}(2n-1)}{N}\right)&=&0\label{tCat1}\\
\lim_{N\to\infty}\varphi_N\left(\frac{W_{N,\rho}(1)}{N}\,\frac{W_{N,\rho}(2)}{N}\ldots \frac{W_{N,\rho}(2n)}{N}\right)&=&\rho^{n} \sum_{\mathscr V\in NC_2(2n)} \rho^{2\,\nest(\mathscr V)}.\label{tCat2}
\end{eqnarray}
In particular, for $\rho\in(0,1]$, 
$$\lim_{N\to\infty}\varphi_N\left(\frac{W_{N,\rho}(1)}{N}\frac{W_{N,\rho}(2)}{N}\ldots \frac{W_{N,\rho}(n)}{N}\right)=\varphi_{0,t}(s_{0,t}(h)^n),$$
for $t=\rho^2$ and $\|h\|_\H=\sqrt{\rho}$.
\end{prop}
\begin{proof}
Expressions \eqref{tCat1} and \eqref{tCat2} can be recovered from \cite{Khorunzhy,Mazza2002} (cf. Theorem 1 \cite{Mazza2002}) via Lemma~\ref{lemtsemicircular}. The following self-contained sketch is included for completeness. First note that the general form of both expressions follows analogously to Wigner's proof of the Semicircle Law. In particular, unrolling the normalized trace $\varphi_N$,
\begin{equation}\varphi_N(W_{N,\rho}(1)W_{N,\rho}(2)\ldots W_{N,\rho}(n))=\frac{1}{N}\sum_{i_1,\ldots,i_n\in[N]}\mathbb E(w_{i_1,i_2}(1)w_{i_2,i_3}(2)\ldots w_{i_n,i_1}(n)).\label{traceunroll}\end{equation}
For a fixed choice of $i_1,\ldots,i_n\in[N]$, the typical argument then proceeds by considering index pairs $\{i_j,i_{j+1}\}$ which repeat and viewing the corresponding pattern as a partition of $[n]$. Since the individual elements $w_{i,j}(k)$ are centered, any partition containing a singleton, i.e. a block formed by a single element, does not contribute to the sum. Next, by counting all the choices of indices $i_1,\ldots,i_n\in[N]$ corresponding to a given partition and taking into account the normalization factors, it can be shown that only the non-crossing pair partitions contribute in the limit. (Note that the counting argument is warranted by the fact that there are uniform bounds, in the time variable $n$, on the higher moments.) This yields \eqref{tCat1} and the general form of \eqref{tCat2}. 

Considering the left-hand side of \eqref{traceunroll}, note that for any index choice  $i_1,\ldots,i_n\in[N]$ such that the repetitions in $\{i_1,i_2\},\{i_2,i_3\},\ldots, \{i_n,i_1\}$ encode a pair-partition, the expectation $E(w_{i_1,i_2}(1) w_{i_2,i_3}(2)$  $\ldots w_{i_n,i_1}(n))$ factors into second moments. Furthermore, recalling that for all $i,j$, $\E(w_{i,j}(k)w_{j,i}(m))=\rho^{m-k}$, one obtains that  
$$\lim_{N\to\infty}\varphi_N\left(\frac{W_{N,\rho}(1)}{N}\,\frac{W_{N,\rho}(2)}{N}\ldots \frac{W_{N,\rho}(2n)}{N}\right)=\sum_{\mathscr V\in NC_2(2n)} \rho^{b_1-a_1}\ldots\rho^{b_n-a_n},$$
where $\mathscr V=\{(a_1,b_1),\ldots,(a_n,b_n)\}$ with $1=a_1<\ldots<a_n$ and $a_i<b_i$. 

It remains to rewrite the above sum in terms of nestings. But, that $\mathscr V$ is non-crossing implies that for any $i\in[n]$, $\rho^{b_i-a_i}=\rho^{1+2\times \text{nest}(a_i,b_i;\mathscr V)}$, where $\text{nest}(a_i,b_i;\mathscr V)$ denotes the number of nestings that include the pair $(a_i,b_i)$, i.e. $\text{nest}(a_i,b_i;\mathscr V):=|\{j\in [n]\mid a_i<a_j<b_j<b_i\}|$. Thus, for any $\mathscr V\in NC_2(2n)$, it follows that $\rho^{b_1-a_1}\ldots\rho^{b_n-a_n}=\rho^{n+2\,\nest(\mathscr V)}$. This yields \eqref{tCat2} and completes the sketch.
\end{proof}


\begin{remark} Taking a process formed from a copy of the same matrix, viz. $W_{N,1}(n)=W_{N}$ a.e. for all $n\in\mathbb N$, yields $\rho=1=t$ and recovers the result of Wigner. In particular, denoting $\tilde W_{N}=W_{N}/N$, the $t$-Catalan numbers become the (usual) Catalan numbers and 
$$\lim_{N\to\infty}\varphi_N\left(\frac{W_{N,1}(1)}{N}\,\frac{W_{N,1}(2)}{N}\ldots \frac{W_{N,1}(n)}{N}\right)=\lim_{N\to\infty}\varphi_N(\tilde W_{N}\,^n)=\varphi_{0,1}(s_{0,1}\,^n),$$
i.e. $\{\lim_{N\to\infty}\varphi_N(\tilde W_{N}\,^n)\}_{n\in\mathbb N}$ is the moment sequence of the semicircular element in free probability. However, for $\rho\in(0,1]$, it is not obvious that $\{\lim_{N\to\infty}\varphi_N(\tilde W_{N,\rho}(1)\,\tilde W_{N,\rho}(2)\ldots \tilde W_{N,\rho}(n))\}_{n\in\mathbb N}$ should still form a moment sequence. Yet, this is indeed the case, and the moments turn out to be those of a $t$-semicircular element with $t=\rho^2$. A deeper principle underlying this fact is presently unclear. Nevertheless, given their ties to  an array of fascinating algebraic and combinatorial objects, their appearance in relation to limits of Wigner processes, and their fundamental role in the generalized non-commutative Central Limit Theorem \cite{Blitvic2}, the $(q,t)$-Gaussians may harbor an additional potential for capturing a broader range of behaviors.
\end{remark}

{\bf Acknowledgements} \quad The author is grateful for the encouragement and advice received from Philippe Biane, Todd Kemp, and Roland Speicher. Thanks are also due to Michael Anshelevich, for pointers towards the physics literature and reference \cite{Bozejko2006} to related work.

\bibliographystyle{ieeetr}
\bibliography{qtFock}
\end{document}